\documentclass[11pt]{article}
\usepackage[utf8]{inputenc}

\usepackage{amsthm, amsfonts, amssymb, amsmath,enumitem, natbib, bbm, mathabx}
\usepackage{cases, amsthm}
\usepackage{fullpage}
\usepackage{mathtools}
\mathtoolsset{showonlyrefs}
\usepackage{wasysym} 
\usepackage{fancyhdr}
\usepackage{graphicx}
\usepackage[ruled,linesnumbered, vlined, noend]{algorithm2e}
\usepackage{bbm}
\usepackage{caption}
\usepackage{subcaption}
\usepackage{authblk}
 \usepackage{setspace}

\RequirePackage[colorlinks=true, pdfstartview=FitV,
linkcolor=blue, citecolor=blue, urlcolor=blue]{hyperref}

\newtheorem{definition}{Definition}[section]
\newtheorem{theorem}[definition]{Theorem}
\newtheorem{proposition}[definition]{Proposition}
\newtheorem{lemma}[definition]{Lemma}

\newcommand{\medbias}{\mathrm{Med}\mbox{-}\mathrm{bias}}
\newcommand{\Rmedbias}{\mathrm{R}\mbox{-}\mathrm{Med}\mbox{-}\mathrm{bias}}
\newcommand{\Tmedbias}{\mathrm{T}\mbox{-}\mathrm{Med}\mbox{-}\mathrm{bias}}
\newcommand{\Omedbias}{\mathrm{O}\mbox{-}\mathrm{Med}\mbox{-}\mathrm{bias}}
\newcommand{\RHull}{\mathrm{Rect}\mbox{-}\mathrm{Hull}}

\newcommand{\MCH}{\mathrm{MCH}}

\usepackage[top=1in, bottom=1in, left=1in, right=1in]{geometry}
\title{Rectangular Hull Confidence Regions for Multivariate Parameters}

\author[1]{Aniket Jain}
\author[2]{Arun Kumar Kuchibhotla}
\affil[1]{Indian Statistical Institute, Kolkata}
\affil[2]{Department of Statistics \& Data Science, Carnegie Mellon University}
\date{}

\begin{document}

\maketitle

\begin{abstract}
We introduce three notions of multivariate median bias, namely, rectilinear, Tukey, and orthant median bias. Each of these median biases is zero under a suitable notion of multivariate symmetry. We study the coverage probabilities of rectangular hull of $B$ independent multivariate estimators, with special attention to the number of estimators $B$ needed to ensure a miscoverage of at most $\alpha$. It is proved that for estimators with zero orthant median bias, we need $B\ge c\log_2(d/\alpha)$ for some constant $c > 0$. Finally, we show that there exists an asymptotically valid (non-trivial) confidence region for a multivariate parameter $\theta_0$ if and only if there exists a (non-trivial) estimator with an asymptotic orthant median bias of zero.
\end{abstract}

\section{Introduction}
Inference from data is the cornerstone of Statistics. There exist numerous methods in the statistical literature for constructing confidence intervals for parameters that are either finite-sample or large-sample (i.e., asymptotically) valid. Some of the most prominent methods include Wald intervals, jackknife, bootstrap, and subsampling. Each of these methods is applicable and/or valid under a variety of assumptions on the data and the estimator used. For a large class of parameters and under mild assumptions, no finite-sample valid, bounded confidence intervals can be constructed~\citep{bahadur1956nonexistence,tibshirani1988sensitive,pfanzagl1998nonexistence}. Under a variety of assumptions, the Wald and resampling based confidence intervals are shown to be asymptotically valid, both in pointwise and uniform sense~\citep{dasgupta2008asymptotic,shao2012jackknife}. The assumptions can be very varied. For example, Wald intervals are asymptotically valid assuming that the rate of convergence and the limiting distribution of the estimator are known (potentially up to an unknown parameter). Bootstrap is valid if the functional of interest is Hadamard differentiable~\citep{shao2012jackknife}. Subsampling is pointwise valid as long as the rate of convergence can be consistently estimated~\citep{bertail1999subsampling}. There exist several examples where all these methods can fail to be uniformly valid over a class of data-generating processes~\citep{andrews2010asymptotic,hirano2012impossibility}. 

There exist, however, numerous inference problems where all these methods fail to provide valid confidence intervals. In constrained M-estimation problems or parameter on the boundary problems, these methods are either inapplicable or do not yield uniform validity. Another example is when the rate of convergence is bounded with sample size. In a recent paper~\citep{kuchibhotla2023median}, it was shown that for any functional of interest for which asymptotically valid non-trivial confidence intervals can be constructed, there exists an asymptotically {\em median unbiased} estimator. Furthermore, the HulC methodology of~\cite{kuchibhotla2021hulc} is known to be asymptotically valid when constructed using an asymptotically median unbiased estimator. Therefore, an asymptotically valid confidence interval exists if and only if there exists an estimator with respect to which HulC is asymptotically valid. To our knowledge, this is not true for the traditional methods of Wald and resampling. The HulC methodology is rather simple to implement. For a parameter/functional of interest $\theta_0$, let $\widehat{\theta}_n$ be the estimator based on $n$ observations. Split the $n$ observations randomly into $B_{\alpha} = \lceil\log_2(2/\alpha)\rceil$ batches of approximately equal sample sizes. Then find the version of $\widehat{\theta}_n$ on each batch and return the convex hull of these univariate estimators as the confidence interval. In the multivariate case, one can consider either the convex hull or the rectangular hull with more number of batches. The validity results are obtained using the union bound, which can be sub-optimal. 

The results of~\cite{kuchibhotla2023median} on the equivalence of valid inference and the existence of asymptotically median unbiased estimators are also restricted to one-dimensional functionals of interest. It is a non-trivial problem to extend the results of that work to the multivariate case given that several notions of multivariate median exist. This is precisely the problem we tackle in this paper. We define a notion of multivariate median bias and show that a locally bounded confidence region exists if and only if there exists an estimator that is asymptotically ``median unbiased.'' The notion of multivariate median bias we consider is based on the median unbiasedness definition used in~\cite{sung1990generalized}. One disadvantage of this notion of multivariate median bias is that even some centrally symmetric distributions are biased. For example, $N(0,\Sigma)$ with a non-diagonal covariance matrix $\Sigma$ turns out have a non-zero median bias.

The remaining article is organized as follows. In Section~\ref{sec:multivariate-median-bias}, we discuss two notions of multivariate median bias, namely, orthant median bias and half-space/Tukey median bias. Both these notions become the univariate median bias discussed in~\cite{kuchibhotla2021hulc}. In Section~\ref{sec:orthat-med-bias}, we restrict our discussion to orthant median bias for multivariate random variables whose marginal distributions are continuous at zero and the construction of confidence regions for arbitrary multivariate parameters based on estimators with bounded orthant median bias. We also discuss the definition of orthant median bias for general distributions that can have a point mass on one or more coordinate axes. In Section~\ref{sec:OMR-valid-inference}, we show one of our prominent results that bounded asymptotically valid confidence regions for a parameter exist if and only if there exists a non-trivial estimator with asymptotic orthant median bias converging to zero asymptotically. We end the paper with Section~\ref{sec:conclusions} by summarizing the article and discussing some future directions.
\section{Multivariate median bias}\label{sec:multivariate-median-bias}
Given the natural ordering of real numbers, the sample median (or the center of the univariate data cloud) is uniquely defined and can be obtained using several equivalent characterizations~\citep{Oja2013}. Each of these characterizations leads to a different vector in the multidimensional space. Using distribution functions, a univariate estimator $\widehat{\theta}_n$ computed based on observations $X_1, \ldots, X_n$ from $P$ is said to have a median $\theta_0$ if 
\[
\min\left\{\mathbb{P}_P(\widehat{\theta}_n \ge \theta_0),\,\mathbb{P}_P(\widehat{\theta}_n \le \theta_0)\right\} \ge \frac{1}{2}.
\]
This can be equivalently written as $\min_{s\in\{-1, 0, 1\}}\mathbb{P}(s(\widehat{\theta}_n - \theta_0) \ge 0) \ge 1/2$,
where we think of $\{-1, 0, 1\}$ as the possible signs of numbers on the real line (obtained by considering $x/|x|$). Correspondingly, the median bias of the estimator $\widehat{\theta}_n$ about $\theta_0$ is defined as
\[
\medbias_{P}(\widehat{\theta}_n; \theta_0) = \left(\frac{1}{2} - \min_{s\in\{-1, 0, 1\}}\mathbb{P}_P(s(\widehat{\theta}_n - \theta_0) \ge 0)\right)_+.
\]
(We reserve $P$ for the distribution of the underlying data $X_1, \ldots, X_n$.)
There are several versions of sign vectors in the multivariate space each of which lead to different versions of median and median bias. For example, in $\mathbb{R}^d$, consider
\begin{align*}
\mathcal{S}_{\texttt{c-sign}} &:= \{s\in\mathbb{R}^d:\, s_{j_0} \in \{-1, 1\}\mbox{ for some }j_0\mbox{ and }s_j = 0\mbox{ for }j\neq j_0\},\\
&= \{s\in\mathbb{R}^d:\, s = \eta e_j, \eta\in\{-1, 1\}, 1\le j\le d\},\\
\mathcal{S}_{\texttt{d-sign}} &:= \{s/\|s\|_2:\, s\in\mathbb{R}^d\setminus\{0\}\}\cup\{0\} = \mathbb{S}^{d-1}\cup\{0\}.
\end{align*}
The set $\mathcal{S}_{\texttt{c-sign}}$ is the set of all coordinate signs, one coordinate at a time, and the set $\mathcal{S}_{\texttt{d-sign}}$ is the set of all directions in $\mathbb{R}^d$ (i.e., the unit sphere). Here $\|\cdot\|_2$ denotes the Euclidean norm of the vector. For any given set $\mathcal{S}$ of sign vectors, we can define the multivariate median of a random vector $X$ as the vector $\theta_{0,\mathcal{S}}$ that maximizes $\min_{s\in\mathcal{S}}\mathbb{P}(s^{\top}X \ge 0)$; if no unique maximizer exists, then any of the maximizers can be used. In particular, with $\mathcal{S} = \mathcal{S}_{\texttt{c-sign}}$, $\theta_{0,\mathcal{S}}$ is the coordinate-wise median or the rectilinear median~\citep{basu2012projection}. With $\mathcal{S} = \mathcal{S}_{\texttt{d-sign}}$, $\theta_{0,\mathcal{S}}$ is the half-space or Tukey median~\citep{zhu2020does}. The corresponding median biases becomes
\begin{equation}\label{eq:R-med-bias}
\begin{split}
\medbias_{P,\mathcal{S}_{\texttt{c-sign}}}(\widehat{\theta}_n; \theta_0) &:= \left(\frac{1}{2} - \min_{s\in\mathcal{S}_{\texttt{c-sign}}}\mathbb{P}(s^{\top}(\widehat{\theta}_n - \theta_0) \ge 0)\right)_+\\ 
&= \left(\frac{1}{2} - \min_{\eta\in\{-1, 1\},1\le j\le d}\mathbb{P}(\eta e_j^{\top}(\widehat{\theta}_n - \theta_0) \ge 0)\right)_+,
\end{split}
\tag{$\Rmedbias$}
\end{equation}
\begin{equation}\label{eq:T-med-bias}
\begin{split}
\medbias_{P,\mathcal{S}_{\texttt{d-sign}}}(\widehat{\theta}_n; \theta_0) &:= \left(\frac{1}{2} - \min_{s\in\mathcal{S}_{\texttt{d-sign}}}\mathbb{P}(s^{\top}(\widehat{\theta}_n - \theta_0) \ge 0)\right)_+,\\
&= \left(\frac{1}{2} - \min_{\lambda\in\mathbb{R}^d}\mathbb{P}(\lambda^{\top}(\widehat{\theta}_n - \theta_0) \ge 0)\right)_+.
\end{split}
\tag{$\Tmedbias$}
\end{equation}
Given the corresponding medians, we refer to $\medbias_{P,\mathcal{S}_{\texttt{c-sign}}}$ as the {\em rectilinear median bias}, or simply, R-median bias and $\medbias_{P,\mathcal{S}_{\texttt{d-sign}}}$ as the {\em Tukey median bias}, or simply, T-median bias. Correspondingly, we sometimes write $\Rmedbias_P(\widehat{\theta}_n; \theta_0)$ in place of $\medbias_{P,\mathcal{S}_{\texttt{c-sign}}}(\widehat{\theta}_n; \theta_0)$ and $\Tmedbias_P(\widehat{\theta}_n; \theta_0)$ in place of $\medbias_{P,\mathcal{S}_{\texttt{d-sign}}}(\widehat{\theta}_n; \theta_0)$.
\cite{kuchibhotla2021hulc} discussed potential extensions of HulC using estimators with controlled median biases in either~\eqref{eq:R-med-bias} or~\eqref{eq:T-med-bias}. Their results imply that for $\mathcal{S}_{\texttt{c-sign}}$, the rectangular hull of sufficient number (approximately $\log(d)$ many) of estimators achieves asymptotically valid coverage, and for $\mathcal{S}_{\texttt{d-sign}}$, the convex hull of sufficient number (approximately $d$ many) of estimators achieves asymptotically valid coverage. 

Another class of multivariate medians can be defined by viewing $[0, \infty)$ and $(-\infty, 0]$ as orthants in the real line. From this viewpoint, we define the {\em orthant median bias} as 
\begin{equation}\label{eq:O-med-bias}
\begin{split}
\medbias_{P,\texttt{orth}}(\widehat{\theta}; \theta_0) 
~&=~ 2^{d-1}\left(\frac{1}{2^d} - \min_{A\in\mathcal{O}^{d}}\mathbb{P}(\widehat{\theta} - \theta_0\in A)\right)_+.
\end{split}
\tag{$\Omedbias$}
\end{equation}
\noeqref{eq:O-med-bias}
if $\widehat{\theta}$ has a continuous distribution. Here $\mathcal{O}^d$ represents the collection of orthants in $\mathbb{R}^d$. We use the notation $\Omedbias_P(\widehat{\theta}_n; \theta_0)$ for orthant median bias.
The restriction to continuous distribution is discussed at length in the next section and the extension of the definition of general distributions in Section~\ref{subsec:general-dist}. In this paper, for a given orthant median bias, we consider the rectangular hull of a sufficient number of estimates as the extension of univariate HulC. We show in some sense that the orthant median bias is the right notion for rectangular hull confidence regions by providing lower bounds on the required number of estimators for valid coverage.

Note that similar to multiple notions of multivariate medians, there are several notions of multivariate symmetry~\citep{serfling2006multivariate}. The R-/T-/O-median biases are zero for multivariate distributions satisfying different types of symmetry. For example, R-median bias is zero for all multivariate distributions whose marginals are symmetric around zero; this is a weaker notion of symmetry than central symmetry. The T-median bias is zero for all angular and half-space symmetric distributions; in fact, $\Tmedbias_P(\widehat{\theta}_n; \theta_0) = 0$ if and only if $\widehat{\theta}_n - \theta_0$ is half-space symmetric around zero. (Even though angular symmetry is more restrictive than half-space symmetry, they are equivalent for both continuous distributions and discrete distributions with no mass at the ``center''~\citep[Theorem 2.6]{zuo2000performance}.) Finally, the O-median bias is zero for all sign symmetric distributions. See~\cite{serfling2006multivariate} for these notions of multivariate symmetry and their characterizations. Also, see Section~\ref{appsec:med-bias-symmetry} for a proof of these relations between zero median bias and symmetry. The target values of probabilities $(1/2, 1/2, 1/2^d)$ used in~\eqref{eq:R-med-bias},~\eqref{eq:T-med-bias}, and~\eqref{eq:O-med-bias} are based on the values attained, respectively, by continuous symmetric random variables that are symmetric in the sense of central symmetry, half-space symmetry, and sign symmetry.
\subsection{Comparison of different median biases}
It is clear that all the notions of median bias discussed above become the same if $d = 1$. For $d > 1$, the comparison of these notions is not so obvious. Because $\mathcal{S}_{\texttt{d-sign}}\supsetneq\mathcal{S}_{\texttt{c-sign}}$, it follows that the rectilinear median bias is always smaller than or equal to the Tukey median bias. For continuous random vectors, it can be shown that the rectilinear median bias is also smaller than or equal to the orthant median bias. The following result summarizes these relations; see Section~\ref{appsec:proof-of-prop-relation-med-bias} for a proof.
\begin{proposition}\label{prop:relation-med-biases}
    For any estimator $\widehat{\theta}_n$ and any target $\theta_0$, 
    \[
    \Rmedbias_P(\widehat{\theta}_n; \theta_0) ~\le~ \Tmedbias_P(\widehat{\theta}_n; \theta_0).
    \]
    For any target $\theta_0$ and any estimator $\widehat{\theta}_n\in\mathbb{R}^d$ that is absolutely continuous with respect to the Lebesgue measure on $\mathbb{R}^d$, 
    \[
    \Rmedbias_P(\widehat{\theta}_n; \theta_0) ~\le~ \Omedbias_P(\widehat{\theta}_n; \theta_0).
    \]
\end{proposition}
For the inequality between rectilinear median bias and the orthant median bias, one does not need the full strength of absolute continuity of the distribution of $\widehat{\theta}_n - \theta_0$, it suffices that $\widehat{\theta}_n - \theta_0$ does not have a positive probability of lying on any $s$-dimensional hyperplane for $s < d$.
Interestingly, there is no ordering between the Tukey and the orthant median bias. We show via two simple examples in the bivariate case that no inequality can exist between these median biases, in general. 
\paragraph{Example 1: Zero orthant median bias but non-zero Tukey median bias.} For $d = 2$, take $\theta_0 = 0$, consider $U$ to be uniformly distributed on $\{x\in\mathbb{R}^2:\,\|x\|_2 \le 1\}$, and suppose $X$ be a bivariate random vector independent of $U$ defined by its probability distribution
 \begin{align*}
   \mathbb{P}(X = (1,1)) = 1/4,\quad
   &\mathbb{P}(X = (2,-1)) = 1/4,\\
   \mathbb{P}(X = (-1,2)) = 1/4,\quad
   &\mathbb{P}(X = (-1,-1))= 1/4.
\end{align*}
Let $\widehat{\theta} = X + 0.1U$. Since $U$ is continuously distributed, so is $\widehat{\theta}$.

Because 
\[
\left\{X= (-1,2),\|U\|_2 \leq 1\right\} \subseteq \left\{\widehat{\theta}_1 \in [-1.1,-0.9], \widehat{\theta}_2 \in [1.9,2.1] \right\} \subseteq \left\{\widehat{\theta}_1 \le 0, \widehat{\theta}_2 \ge 0 \right\},
\]
we get that $\mathbb{P}(\widehat{\theta}_1 \le 0, \widehat{\theta}_2 \ge 0) \geq 1/4$. A similar calculation yields the same bound for all quadrants. Hence, $\Omedbias_P(\widehat{\theta}; \theta_0) = 0$.
On the other hand,
\begin{align*}
&\left\{X \in \left\{(1,1),(2,-1),(-1,2)\right\}, \|U\|_2 \leq 1\right\}\\ 
&\quad\subseteq \left\{(1,1)^{\top}\widehat{\theta} = (1,1)^{\top}X + 0.1(1,1)^{\top}U \geq 1 - (0.1) \sqrt{2} > 0 \right\}.
\end{align*}
Thus, $\mathbb{P}((1,1)^{\top}\widehat{\theta} >0) \geq 3/4$ which implies $\mathbb{P}((1,1)^{\top}\theta \leq 0) \leq 1/4$. Hence, $\Tmedbias_P(\widehat{\theta}; \theta_0) \geq 1/4$ (in fact, equality holds).\hfill$\diamond$

\paragraph{Example 2: Zero Tukey median bias but non-zero orthant median bias.}
Let $\nu$ be the probability measure of a standard bivariate normal random vector centered at the origin with identity dispersion. Take $\theta_0= 0$ and consider the following distribution for $\hat{\theta}_n$. Define
\begin{equation}\label{eq:example2}
\mathbb{P}(\hat{\theta}_n \in A) =
\begin{cases}
    3\nu(A)/2,&\mbox{if } A\subseteq\{(x_1, x_2)\in\mathbb{R}^2:\, x_1x_2 \ge 0\}, \\
    \nu(A)/2,&\mbox{if } A\subseteq\{(x_1, x_2)\in\mathbb{R}^2:\, x_1x_2 \le 0\}.
\end{cases}
\end{equation}
It is easy to verify that $\mathbb{P}(\widehat{\theta}_n\in\mathbb{R}^2) = 1$. Using finite additivity of probability measures, we obtain a unique probability measure satisfying~\eqref{eq:example2}. Clearly, (with $\mathcal{O}^2$ representing the collection of orthants in $\mathbb{R}^2$) 
\[
\min_{A\in\mathcal{O}^2}\,\mathbb{P}(\widehat{\theta}_n - \theta_0 \in A) = \frac{1}{4}\min\left\{\frac{3}{2}, \frac{1}{2}\right\} = \frac{1}{8}.
\]
Therefore, $\Omedbias_P(\widehat{\theta}_n; \theta_0) = 2(1/4 - 1/8) = 1/4$. On the other hand, for any $A\subset\mathbb{R}^2$, $\mathbb{P}(\widehat{\theta}_n - \theta_0 \in A) = \mathbb{P}(\theta_0 - \widehat{\theta}_n \in A)$ and $\widehat{\theta}_n - \theta_0$ has an absolutely continuous distribution. Hence, by Proposition~\ref{prop:relation-med-biases}, $\Tmedbias_P(\widehat{\theta}_n; \theta_0) = 0$.\hfill$\diamond$

As mentioned in the introduction, \cite{kuchibhotla2021hulc} proposed the use of the convex hull of independent estimators with controlled Tukey median bias and the rectangular hull of independent estimators with controlled rectilinear median bias. When the Tukey median bias is zero, then for validity the convex hull requires at least $d$ independent estimators which is not feasible in high-dimensional settings. Because the rectangular hull is by definition larger than the convex hull, when rectilinear median bias is controlled, one only needs about $\log(d)$ independent estimators. However, the coverage guarantee is obtained using a union bound and in practice one may end up with a truly conservative confidence region. In the remaining article, we focus on orthant median bias and show that for estimators with controlled orthant median bias, the requirement of $\log(d)$ estimators is strict in the sense that with less than $O(\log(d))$ independent estimators, the rectangular hull may not have valid coverage.
\section{Orthant median bias (OMB)}\label{sec:orthat-med-bias}
In this section, we discuss in detail the definition of orthant median bias and consider the coverage properties of the rectangular hull of independent estimators. Recall that~\eqref{eq:O-med-bias} is only presented for continuous distributions. Mathematically, there is no reason to not use such a definition for all probability distributions. In Section~\ref{subsec:MCH-distributions}, we show the use of~\eqref{eq:O-med-bias} for distributions that place zero mass on canonical coordinate axes (i.e., boundaries of orthants) in constructing confidence regions. In Section~\ref{subsec:general-dist}, we show that a particular statistical intuition breaks when using~\eqref{eq:O-med-bias} as is for all discrete distributions and define the extension of~\eqref{eq:O-med-bias} for general distributions. The statistical intuition we are referring to is the fact that (in the univariate case) the convex hull from any estimator that places non-zero mass at $\theta_0$ has higher coverage (or equivalently, lower miscoverage) than the convex hull from an estimator that places zero mass at $\theta_0$, provided that both estimators place the same mass on $(-\infty, 0)$ and $(0, \infty)$. Similarly, we would expect that miscoverage from hulls would be smaller if the estimator places non-zero mass on any of the canonical coordinate axes. Although an appealing property, this may not be satisfied for orthant median bias as defined in~\eqref{eq:O-med-bias}, as shown in Section~\ref{subsec:general-dist}. 
\subsection{Orthant median bias for MCH distributions}\label{subsec:MCH-distributions}
Let us call a distribution that puts zero mass on all canonical coordinate axes an MCH distribution. More formally, we write $Q\in\MCH$ for a probability measure $Q$ if $Q(\{x\in\mathbb{R}^d:\, \|x\|_0 < d\}) = 0$, where $\|x\|_0 = \sum_{j=1}^d \mathbf{1}\{x_j \neq 0\}$. If a random vector $W$ has a probability measure $Q$, then we write $W\in\MCH$ to denote $Q\in\MCH.$ There is a rather simple characterization of the MCH distributions in terms of the continuity of the marginal distributions. 
\begin{proposition}\label{prop:MCH-characterization}
For any random vector $X\in\mathbb{R}^d$, $X\in\MCH$ if and only if $\mathbb{P}(e_j^{\top}X = 0) = 0$ for all $1\le j\le d$ (or equivalently, the distribution function of $e_j^{\top}X$ is continuous at zero for all $1\le j\le d$).
\end{proposition}
See Appendix~\ref{appsec:proof-of-prop-MCH-characterization} for a proof. Proposition~\ref{prop:MCH-characterization} implies that all absolutely continuous distributions are MCH. As mentioned, we define orthant median bias of $\widehat{\theta}_n$ about $\theta_0$ as in~\eqref{eq:O-med-bias} if $\widehat{\theta}_n - \theta_0\in\MCH$: 
\[
\Omedbias_P(\widehat{\theta}_n; \theta_0) ~:=~ 2^{d-1}\left(\frac{1}{2^d} - \min_{A\in\mathcal{O}^d}\mathbb{P}(\widehat{\theta}_n - \theta_0\in A)\right)_+,\quad\mbox{if }\widehat{\theta}_n - \theta_0\in\MCH.
\] 
Because MCH distributions do not place any mass on the coordinate axes,  $\Omedbias_P(\widehat{\theta}_n; \theta_0) = 0$ for $\widehat{\theta}_n-\theta_0\in\MCH$ if and only if $\mathbb{P}(\widehat{\theta}_n - \theta_0\in A) = 1/2^d$ for all $A\in\mathcal{O}^d$. In words, orthant median bias is zero if and only if all orthants have an equal probability of $1/2^d$.

We are now ready to present one of the main results of this section on the miscoverage of rectangular hull of independent estimators. If $\widehat{\theta}^{(1)}, \ldots, \widehat{\theta}^{(B)}$ are independent estimators in $\mathbb{R}^d$, then define the rectangular hull of these estimators as
\[
\RHull_B 
~:=~ \bigotimes_{j=1}^d \left[\min_{1\le k\le B}\,e_j^{\top}\widehat{\theta}^{(k)},\, \max_{1\le k\le B}\,e_j^{\top}\widehat{\theta}^{(k)}\right].
\]
The following result provides upper and lower bounds for the probability of not covering $\theta_0$, $\mathbb{P}(\theta_0\notin\RHull_B)$, for independent estimators $\widehat{\theta}^{(k)}, 1\le k\le B$.
\begin{theorem}\label{thm:miscoverage-MCH}
    If $\widehat{\theta}^{(k)}, 1\le k\le B$ are independent and identically distributed random vectors in $\mathbb{R}^d$ and $\widehat{\theta}^{(1)} - \theta_0\in\MCH$, then setting $\delta := \Omedbias_P(\widehat{\theta}^{(1)};\theta_0),$ we have
    \[
    L(B, \delta; d) ~\le~ \mathbb{P}\left(\theta_0\notin\RHull_B\right) ~\le~ U(B, \delta; d),
    \]
    where
    \begin{align*}
        L(B, \delta; d) &:= \sum_{\substack{1\leq j \leq d,\\ j\; \mathrm{odd}}}\binom{d}{j}2^j \left(\frac{1}{2^j}\right)^B \\
        &\qquad- \sum_{\substack{1\leq j \leq d,\\ j\; \mathrm{ even}}}
  \binom{d}{j}\left((2^j - 1)\cdot \left(\frac{1}{2^j} - \frac{\delta}{2^{j-1}}\right)^B + \left(\frac{1}{2^j} + (2^j - 1)\cdot\frac{\delta}{2^{j-1}}\right)^B\right),\\
        U(B, \delta; d) &:= \sum_{\substack{1\leq j \leq d,\\ j\; \mathrm{ odd}}}\binom{d}{j}\left((2^j - 1)\cdot \left(\frac{1}{2^j} - \frac{\delta}{2^{j-1}}\right)^B + \left(\frac{1}{2^j} + (2^j - 1)\cdot\frac{\delta}{2^{j-1}}\right)^B\right)\\
        &\qquad- \sum_{\substack{1\leq j \leq d,\\ j\; \mathrm{ even}}} \binom{d}{j}2^j \left(\frac{1}{2^j}\right)^B.
    \end{align*}
\end{theorem}
See Section~\ref{appsec:proof-of-thm-miscoverage-MCH} for a proof of Theorem~\ref{thm:miscoverage-MCH}. If $d = 1$, then
\[
L(B, \delta; 1) = \frac{1}{2^{B-1}}\quad\mbox{and}\quad U(B, \delta; 1) = \left(\frac{1}{2} - \delta\right)^B + \left(\frac{1}{2} + \delta\right)^B.
\]
Therefore, for $d = 1$, Theorem~\ref{thm:miscoverage-MCH} reduces to Lemma 1 and Theorem 2 of~\cite{kuchibhotla2021hulc}. In other words, the result matches with the known one for univariate HulC. It is worth stressing here that a non-trivial (i.e., non-zero) lower bound on the miscoverage probability as presented in Theorem~\ref{thm:miscoverage-MCH} is possible only because we are restricting to MCH distributions. For example, if $\widehat{\theta}^{(1)}$ is degenerate at $\theta_0$, then the miscoverage probability is zero. However, such a degenerate distribution is not MCH, and Theorem~\ref{thm:miscoverage-MCH} does not apply.

{\begin{algorithm}[!t]
    \caption{Rectangular Hull Confidence Region}
    \label{alg:confidence-zero-med-bias}
    \SetAlgoLined
    \SetEndCharOfAlgoLine{}
    \KwIn{data $X_1, \ldots, X_n$, coverage probability $1 - \alpha$, and an estimation procedure $\mathcal{A}(\cdot)$ that takes as input observations and returns an estimator.}
    \KwOut{A confidence interval $\widehat{\mathrm{CI}}_{\alpha}$ such that $\mathbb{P}(\theta_0\in\widehat{\mathrm{CI}}_{\alpha}) \ge 1 - \alpha + o(1)$ as $n\to\infty$.}
    Set
    $B_{\alpha,d} := \left\lceil1 - \log_2\left({1 - (1-\alpha)^{1/d}}\right)\right\rceil,$ so that $U(B_{\alpha,d}, 0; d) \le \alpha$\;
    
    Generate a random variable $U$ from the Uniform distribution on $[0,1]$ and set
    \begin{equation}\label{eq:definition-tau}
    \tau_{\alpha,d} := \frac{\alpha - U(B_{\alpha,d}, 0; d)}{U(B_{\alpha,d} - 1, 0; d) - U(B_{\alpha,d}, 0; d)}\quad\mbox{and}\quad B^* := \begin{cases} B_{\alpha,d} - 1, &\mbox{if }U \le \tau_{\alpha,d},\\
    B_{\alpha,d}, &\mbox{if }U > \tau_{\alpha,d}.
    \end{cases} 
    \end{equation}\;
    
    Randomly split the data $X_1, \ldots, X_n$ into $B^*$ disjoint sets $\{\{X_i:i\in S_j\}: 1\le j\le B^*\}$, all of equal sizes. Discard some observations, if needed.\;
    
    Compute estimators $\widehat{\theta}^{(j)} := \mathcal{A}(\{X_i:\,i\in S_j\})$, for $1\le j\le B^*$\;
    
    \Return the rectangular hull confidence region $\widehat{\mathrm{CI}}_{\alpha} ~:=~ \RHull_{B^*}$.
\end{algorithm}}

Some useful properties of $U(\cdot, \cdot; d)$ and $L(\cdot, \cdot; d)$ are noted in the following proposition (proved in Section~\ref{appsec:proof-of-prop-properties-of-L-and-U}).
\begin{proposition}\label{prop:properties-of-L-and-U}
The following statements hold true.
\begin{enumerate}
    \item For any $B, d \ge 1$, $\delta\mapsto U(B, \delta; d)$ is a continuous increasing function.
    \item For any $B, d \ge 1$,
    \[
    L(B, 0; d) = U(B, 0; d) = 1 - \left(1 - 2^{-B+1}\right)^d.
    \]
    \item For any $B, d\ge1$, 
    \[
    \frac{d}{d\delta}U(B, \delta; d)\bigg|_{\delta = 0} = 0,\quad\mbox{and}\quad \frac{d^2}{d\delta^2}U(B,\delta;d) \le 2^dB(B-1),\quad\mbox{for}\quad \delta\in[0, 1/2].
    \]
\end{enumerate}
\end{proposition}
Based on Theorem~\ref{thm:miscoverage-MCH} and Proposition~\ref{prop:properties-of-L-and-U}, we propose a rectangular hull confidence region (in Algorithm~\ref{alg:confidence-zero-med-bias}) for estimators that have an asymptotic orthant median bias of zero. The following theorem (proved in Section~\ref{appsec:proof-of-asymp-validity}) shows that for MCH estimators that have an asymptotic orthant median bias of zero, the rectangular hull confidence region of Algorithm~\ref{alg:confidence-zero-med-bias} is asymptotically valid.
\begin{theorem}\label{thm:asymptotic-validity}
    Suppose $X_1, \ldots, X_n$ are independent and identically distributed from $P$, and the estimators $\widehat{\theta}^{(j)}, 1\le j\le B^*$ (as in Algorithm~\ref{alg:confidence-zero-med-bias}) satisfy
    \[
    \widehat{\theta}^{(j)} - \theta_0 \in\MCH,\quad\mbox{and}\quad \Omedbias_P(\widehat{\theta}^{(j)}; \theta_0) \le \delta_{n/B_{\alpha,d}},
    \]
    then 
    \[
    \alpha - 2^{d-1}B_{\alpha,d}(B_{\alpha,d} - 1)\delta_{n/B_{\alpha,d}}^2 ~\le~ \mathbb{P}(\theta_0 \notin \widehat{\mathrm{CI}}_{\alpha}) ~\le~ \alpha + 2^{d-1}B_{\alpha,d}(B_{\alpha,d} - 1)\delta_{n/B_{\alpha,d}}^2. 
    \]
\end{theorem}

The median bias bound in Theorem~\ref{thm:asymptotic-validity} is indexed with $n/B_{\alpha,d}$ because the estimators $\widehat{\theta}^{(j)}$ are computed using $\lfloor{n/B^*}\rfloor\asymp n/B_{\alpha,d}$ independent observations.
Theorem~\ref{thm:asymptotic-validity} proves that as $\delta_{n/B_{\alpha,d}}$ converges to zero, the rectangular hull confidence region of Algorithm~\ref{alg:confidence-zero-med-bias} has a miscoverage converging to $\alpha$, i.e., $\widehat{\mathrm{CI}}_{\alpha}$ is asymptotically valid with an asymptotic confidence of $1-\alpha$. Convergence of $\delta_{n/B_{\alpha,d}}$ to zero holds true if there exists a sequence $\{r_{n/B^*}\}_{n\ge1}$ such that $r_{n/B^*}(\widehat{\theta}^{(j)} - \theta_0)\overset{d}{\to}W$ for some sign symmetric random vector $W\in\MCH$ (e.g., $N(0, I_d)$). This follows because orthants are continuity sets for any MCH distribution and convergence in distribution implies convergence of probabilities of continuity sets. Therefore, convergence in distribution to an MCH distribution implies convergence of orthant median biases as well. The most interesting aspect of Theorem~\ref{thm:asymptotic-validity} is that the rectangular hull confidence region is second order accurate in that the miscoverage is close to $\alpha$ at the rate of squared median bias. A similar phenomenon is known in the one-dimesional case for HulC~\citep{kuchibhotla2021hulc}. However, their bound is sharper with a multiplicative error, while ours is additive. 

From our characterization of MCH distributions (Proposition~\ref{prop:MCH-characterization}), it follows that the condition that $\widehat{\theta}^{(j)} - \theta_0\in\MCH$ holds true if $\widehat{\theta}^{(j)}$ has an absolutely continuous distribution. This is, however, a restrictive condition to guarantee in practice. In the following subsection, we show that the same upper bound on the miscoverage continuous to hold true for non-MCH estimators as well but with a more complicated definition of orthant median bias. 
\subsection{Orthant median bias for general distributions}\label{subsec:general-dist}
Before proceeding to our definition of orthant median bias for non-MCH distributions, let us present an example to show why definition~\eqref{eq:O-med-bias} is not suitable for non-MCH distributions.
\paragraph{Example: Deficiency of~\eqref{eq:O-med-bias} for non-MCH distributions.} As described in the beginning of Section~\ref{sec:orthat-med-bias}, we would expect two distributions with same median bias to have an ordering in terms of the micoverage probability that the MCH distribution has higher miscoverage than the non-MCH distribution. In this example, we exhibit two distribution, one MCH and the other non-MCH, but the non-MCH distribution has a higher miscoverage probability for the rectangular hull than that of the MCH distribution. Take two probability measures $\mu_{\texttt{M}}$ (an MCH distribution) and $\mu_{\texttt{NM}}$ (a non-MCH distribution) such that
\begin{align*}
    \mu_{\texttt{M}}(\{x_1 \ge 0, x_2 \ge 0\}) ~&=~ \mu_{\texttt{NM}}(\{x_1 \ge 0, x_2 \ge 0\}) ~=~ 0.2,\\
    \mu_{\texttt{M}}(\{x_1 \ge 0, x_2 \le 0\}) ~&=~ \mu_{\texttt{NM}}(\{x_1 \ge 0, x_2 \le 0\}) ~=~ 0.2,\\
    \mu_{\texttt{M}}(\{x_1 \le 0, x_2 \ge 0\}) ~&=~ \mu_{\texttt{NM}}(\{x_1 \le 0, x_2 \ge 0\}) ~=~ 0.2,\\
    \mu_{\texttt{M}}(\{x_1 \le 0, x_2 \le 0\}) &= 0.6,\quad \mu_{\texttt{NM}}(\{x_1 \le 0, x_2 \le 0\}) = 0.4.
\end{align*}
In words, the first three quadrants get a probability mass of $0.2$ and the fourth quadrant (negative $x$ and negative $y$-axis) gets a probability mass of $0.4$ and $0.6$, respectively, for the MCH and non-MCH distributions. Because the minimum mass of closed orthants is $0.2$ for both distributions, the orthant median bias as defined in~\eqref{eq:O-med-bias} is $2(0.25 - 0.2) = 0.1$ for both distributions. For the MCH distribution $\mu_{\texttt{M}},$ the probability content of the open orthants are also the same. We define the non-MCH distribution $\mu_{\texttt{NM}}$ to have non-zero point mass on the positive $y$-axis and negative $x$-axis:
\[
\mu_{\texttt{NM}}(\{x_1 = 0, x_2 > 0\}) = 0.1 = \mu_{\texttt{NM}}(\{x_1 < 0, x_2 = 0\}).
\]
Hence, $\mu_{\texttt{NM}}\notin\MCH$. If we now consider three random vectors from each of these distributions, then 
\begin{equation}\label{eq:miscoverage-prob-MCH-non-MCH}
\begin{split}
\mathbb{P}_{\mu_{\texttt{M}}}(0 \notin \RHull_3) = 0.472,\\
\mathbb{P}_{\mu_{\texttt{NM}}}(0 \notin \RHull_3) = 0.484.
\end{split}
\end{equation}
(See Section~\ref{appsec:proof-of-miscoverage-prob-MCH-non-MCH} for a detailed proof.) \hfill$\diamond$

Equalities~\eqref{eq:miscoverage-prob-MCH-non-MCH} show that any bound we derive with~\eqref{eq:O-med-bias} for MCH distributions need not be valid for non-MCH distributions. For this reason, we develop an alternative definition of orthant median bias for non-MCH distributions. The basic intuition for this is to look for all MCH distributions that are obtained from the given non-MCH distribution by dispersing the probability mass on the edges of an orthant to orthants that share that edge and finding the least orthant median bias among all such MCH distributions. 

For ease of presentation, define
\begin{align*}
U_d &:= \{-1, 0, 1\}^d = \{\eta\in\mathbb{R}^d:\, \eta_j\in\{-1, 0, 1\}\mbox{ for all }j\},\mbox{ and}\\
V_d &:= \{-1, 1\}^d = \{\eta\in\mathbb{R}^d:\, \eta_j\in\{-1, 1\}\mbox{ for all }j\}. 
\end{align*}
Note that every element of $V_d$ represents an orthant and every element in $U_d\setminus V_d$ represents a edge for some orthant. If $A$ is an orthant ($A\in\mathcal{O}^d$), then there exists an $\eta\in V_d$ such that for every $x\in A^\circ$ (the interior of $A$), $\mbox{sgn}(x) = \eta$. By flipping one or more of the bits in $\eta$ to zero, we get to an element in $U_d\setminus V_d$ and represent different edges of the orthant $A$.

Define a partial sign ordering on $U_d$ as follows: for $\eta, \eta'\in U_d$,
\[
\eta\preceq_s \eta' \quad\mbox{if and only if}\quad 
\begin{cases}
\eta_j = \eta_j',\mbox{ for all }j\mbox{ such that }\eta_j\neq 0,\\
\eta_j \le |\eta_j'|,\mbox{ for all }j\mbox{ such that }\eta_j = 0.
\end{cases}
\]
This partial ordering captures the idea of points in the open orthants are greater than points on an edge of that orthant and that all points in open orthants are ``equal.'' The sign partial ordering dictates that a zero sign is less than a non-zero sign. Note that an edge of an orthant in $\mathbb{R}^d$ can be shared by many orthants. Finally, for any probability measure $\mu$ on $\mathbb{R}^d$, define a probability measure on $U_d$ as 
\[
\mu_{\texttt{sgn}}(\eta) := \mu(\{x\in\mathbb{R}^d:\,\mbox{sign}(x) = \eta\}),\quad\mbox{for}\quad \eta\in U_d.
\]
Here $\mbox{sign}(x)$ represents a vector in $\mathbb{R}^d$ with $e_j^{\top}\mbox{sign}(x)$ being $-1, 0$, or $1$ according to if $e_j^{\top}x$ is negative, zero, or positive.

With this set-up, we are now ready to define MCH ordering of probability measures via an {\em elementary} operation. We say that a probability measure $\lambda$ is obtained from another probability measure $\mu$ by an elementary operation if there exists $q \ge 0$ and a pair $(\eta, \eta')$ with $\eta\in U_d\setminus V_d$ and $\eta \preceq_s \eta'$ such that
\begin{equation}\label{eq:elementary-operation-definition}
\begin{split}
\mu_{\texttt{sgn}}(\gamma) ~&=~ \lambda_{\texttt{sgn}}(\gamma) \mbox{ for all }\gamma\in U_d\setminus\{\eta, \eta'\},\quad
\mu_{\texttt{sgn}}(\eta) ~\ge~ q,\\
\lambda_{\texttt{sgn}}(\eta) ~&=~ \mu_{\texttt{sgn}}(\eta) - q,\quad
\lambda_{\texttt{sgn}}(\eta') ~=~ \mu_{\texttt{sgn}}(\eta') + q.
\end{split}
\end{equation}
We write $\lambda = \mathcal{E}(\mu)$ if $\lambda$ is obtained from $\mu$ by an elementary operation $\mathcal{E}(\cdot)$. This elementary operation signifies the dispersion operation that moves some mass from one of the edges to another edge at a higher level (i.e., an edge that has a lesser number of zero's in the sign vector). 
\paragraph{Example: Elementary operation in $\mathbb{R}$.} In $\mathbb{R}$, there are only three possible signs $U_1 = \{-1, 0, 1\}$ and two ``orthants'' $V_d = \{-1, 1\}$. Any elementary operation has to move mass from a set with $\{0\}$ sign to a set with either $\{-1\}$ or $\{1\}$ signs, i.e., any elementary operation disperses mass from $\{0\}$ to either $(-\infty, 0)$ or $(0,\infty)$. Suppose $\mu$ is a probability measure on $\mu$ such that 
\[
\mu((0,\infty)) = \mu_{\texttt{sgn}}(1) = 0.2,\quad \mu((-\infty, 0)) = \mu_{\texttt{sgn}}(-1) = 0.4, \quad\mbox{and}\quad \mu(\{0\}) = \mu_{\texttt{sgn}}(0) = 0.4.
\]
Then an elementary operation on $\mu$ takes a probability mass of at most $0.4$ and disperses it to either $(-\infty, 0)$ or $(0, \infty)$. For example, define a probability measure $\lambda$ such that
\[
\lambda_{\texttt{sgn}}(0) = 0.2,\quad \lambda_{\texttt{sgn}}(-1) = 0.6,\quad\mbox{and}\quad \lambda_{\texttt{sgn}}(1) = 0.2. 
\]
We can write $\lambda = \mathcal{E}_1(\mu)$ where the parameters for $\mathcal{E}_1$ are $q = 0.2, \eta = \{0\},$ and $\eta' = \{-1\}$. We can continue applying such elementary operations as long as the probability mass at $\{0\}$ is non-zero. For example, define a probability measure $\nu$ such that
\[
\nu_{\texttt{sgn}}(0) = 0, \quad \nu_{\texttt{sgn}}(-1) = 0.6,\quad\mbox{and}\quad \nu_{\texttt{sgn}}(1) = 0.4.
\]
We can write $\nu = \mathcal{E}_2(\lambda)$ where the parameters for $\mathcal{E}_2$ are $q = 0.2, \eta = \{0\},$ and $\eta' = \{1\}$. We can move further from $\nu$ by elementary operations because there is no mass left at $0$. Note that $\nu\in\MCH$ and we got $\nu = \mathcal{E}_2\mathcal{E}_1(\mu)$.\hfill$\diamond$ 

We now define MCH ordering. For two probability measures $\mu$ and $\nu$, we say $\mu \preceq_{\texttt{MCH}} \nu$ if there exists a finite sequence of elementary operations $\mathcal{E}_1, \ldots, \mathcal{E}_k$ such that $\nu = \mathcal{E}_k\cdots\mathcal{E}_1(\mu)$. From the definition of the elementary operation, it is clear that the MCH ordering only relates to the induced probability measures on the sign vectors. 
We conjecture that MCH ordering can be equivalently be defined as
\begin{equation}\label{eq:MCH-ordering}
\mu\preceq_{\texttt{MCH}} \nu\quad \mbox{if and only if}\quad
\mu_{\texttt{sgn}}(\eta) \le \nu_{\texttt{sgn}}(\eta')\mbox{ for all }\eta, \eta'\in U_d\mbox{ such that }\eta\preceq_{s}\eta'.
\tag{MCH Order}
\end{equation}

With respect to the ordering $\preceq_{\texttt{MCH}}$, MCH distributions are the greatest distributions because MCH distributions do not place any mass on $\{x:\,\mbox{sign}(x) = \eta\}$ for $\eta\in U_d\setminus V_d$ and hence, all elementary operations must take $q=0$. Now define the orthant median bias for any arbitrary probability measure $\mu$ as 
\begin{equation}\label{eq:O-med-bias-general}
\begin{split}
    \texttt{OMB}(\mu) &:= \inf_{\nu\in\MCH:\,\mu\preceq_{\texttt{MCH}}\nu}\, 2^{d-1}\left(\frac{1}{2^d} - \min_{A\in\mathcal{O}^d}\,\nu(A)\right)_+\\
    &:= \inf_{\nu\in\MCH:\,\mu\preceq_{\texttt{MCH}}\nu}\, 2^{d-1}\left(\frac{1}{2^d} - \min_{\eta\in V_d}\nu_{\texttt{sgn}}(\eta)\right)_+.
\end{split}
\end{equation}
If $\mu\in\MCH$, then $\mu\preceq_{\texttt{MCH}}\nu$ implies $\mu_{\texttt{sgn}}(\eta) = \nu_{\texttt{sgn}}(\eta)$ for all $\eta\in V_d$. This is because every elementary operation from $\mu$ to $\nu$ must use $q = 0$. 
This implies that if $\mu\in\MCH$, then $\texttt{OMB}(\mu) = 2^{d-1}(1/2^d - \min_{A\in\mathcal{O}^d}\mu(A))_+$. This matches our definition~\eqref{eq:O-med-bias} for MCH distributions.
For an estimator $\widehat{\theta}_n$, the median bias about $\theta_0$ is defined as
\begin{equation}\label{eq:Orthant-med-bias-general}
\Omedbias_P(\widehat{\theta}_n; \theta_0) := \texttt{OMB}(\mathcal{L}(\widehat{\theta}_n - \theta_0)),
\end{equation}
where $\mathcal{L}(\widehat{\theta}_n - \theta_0)$ is the probability measure of $\widehat{\theta}_n - \theta_0$. 

With the renewed definition~\eqref{eq:Orthant-med-bias-general} of orthant median bias, the upper bound on the miscoverage probability in Theorem~\ref{thm:miscoverage-MCH} holds true without the need for $\widehat{\theta}^{(1)} - \theta_0\in\MCH$. This is shown formally stated in Theorem~\ref{thm:general-miscoverage-MCH-validity} (proved in Section~\ref{appsec:proof-of-general-miscoverage-MCH-validity}). 
\begin{theorem}\label{thm:general-miscoverage-MCH-validity}
    Suppose $X_1, \ldots, X_n$ are independent and identically distributed from $P$, and the estimators $\widehat{\theta}^{(j)}, 1\le j\le B^*$ (as in Algorithm~\ref{alg:confidence-zero-med-bias}) satisfy
    \[
    \Omedbias_P(\widehat{\theta}^{(j)}; \theta_0) = \texttt{OMB}(\mathcal{L}(\widehat{\theta}^{(j)} - \theta_0)) \le \delta_{n/B_{\alpha,d}},
    \]
    then $\mathbb{P}(\theta_0\notin\widehat{\mathrm{CI}}_{\alpha}) \le \alpha + 2^{d-1}B_{\alpha,d}(B_{\alpha,d} - 1)\delta_{n/B_{\alpha,d}}^2$.
\end{theorem}
We prove Theorem~\ref{thm:general-miscoverage-MCH-validity} by showing that for a given $\mu\notin\MCH$ and $\nu\in\MCH$ satisfying $\mu\preceq_{\texttt{MCH}} \nu$, there exists a coupling such that the event of miscoverage of zero of rectangular hull of observations from $\mu$ is a subset of that of observations from $\nu$. This implies that the miscoverage probability under $\mu$ is smaller than the miscoverage probability under $\nu$. Given that this forms the core of the proof of Theorem~\ref{thm:general-miscoverage-MCH-validity}, we present this result as a lemma below. Because $\nu\in\MCH$ and Theorems~\ref{thm:miscoverage-MCH},~\ref{thm:asymptotic-validity} already provide upper bounds on miscoverage for MCH distributions, we obtain a bound on miscoverage for general distributions. 
\begin{lemma}\label{lem:miscoverage-elementary-operation}
    Suppose $\mu$ and $\nu$ are two probability measures such that $\nu = \mathcal{E}(\mu)$. Suppose $(X_1, \ldots, X_B)$ and $(Y_1, \ldots, Y_B)$ represents two independent sets of independent and identically distributed random vectors from $\mu$ and $\nu$, respectively. Then
    \[
    \mathbb{P}\left(0 \notin \bigotimes_{j=1}^d\left[\min_{1\le i\le B}e_j^{\top}X_i,\,\max_{1\le i\le B}e_j^{\top}X_i\right]\right) ~\le~ \mathbb{P}\left(0 \notin \bigotimes_{j=1}^d\left[\min_{1\le i\le B}e_j^{\top}Y_i,\,\max_{1\le i\le B}e_j^{\top}Y_i\right]\right).
    \]
    Moreover, the same conclusion continuous to hold true if $\mu \preceq_{\texttt{MCH}} \nu$.
\end{lemma}

In the remaining part of this section, we study $\texttt{OMB}(\mu)$ and the revised definition~\eqref{eq:Orthant-med-bias-general}. 

An estimator $\widehat{\theta}_n$ that is identically equal to $\theta_0$ (with probability 1) should be unbiased, no matter the notion of unbiasedness. Unfortunately, such an estimator does not belong to $\MCH$. The following proposition shows that an estimator that is identically $\theta_0$ has an orthant median bias of $0$.
\begin{proposition}\label{prop:orthant-median-bias-null-est}
    Suppose $\mathbb{P}(\widehat{\theta}_n = \theta_0) = 1$. Then
    \[
    \Omedbias_P(\widehat{\theta}_n; \theta_0) = 0.
    \]
\end{proposition}
See Section~\ref{appsec:proof-of-orthant-median-bias-null-est} for a proof. The proof essentially applies $2^d$ elementary operations each moving a mass of $1/2^d$ to each of the $2^d$ orthants.

The following proposition shows that in the univariate case, we have an explicit formula for the orthant median bias.
\begin{proposition}\label{prop:orthant-median-bias-1-d}
    If $d = 1$, then for any estimator $\widehat{\theta}_n\in\mathbb{R}$ and $\theta_0\in\mathbb{R}$,
    \[
    \Omedbias_P(\widehat{\theta}_n; \theta_0) = \left(\frac{1}{2} - \min_{\eta\in\{-1, 1\}}\mathbb{P}(\eta(\widehat{\theta}_n - \theta_0) \ge 0)\right)_+.
    \]
\end{proposition}
See Section~\ref{appsec:proof-of-prop-orthant-median-bias-1-d} for a proof. Proposition~\ref{prop:orthant-median-bias-1-d} shows that in the univariate case the orthant median bias defined via MCH distributions matches the definition for MCH distributions and moreover, orthant median bias matches the univariate median bias used in~\cite{kuchibhotla2021hulc,kuchibhotla2023median}.

One of the important aspects for practical performance of the rectangular hull confidence region is to have an asymptotically zero orthant median bias. If $\widehat{\theta}_n - \theta_0\in\MCH$, then by weak convergence, the orthant median bias of the estimators converge to the orthant median bias of the limiting distribution. The following proposition proves that the same result continues to hold true for general distributions. 
\begin{proposition}\label{prop:convergence-in-distribution-OMB}
    Suppose $\widehat{\theta}_n, n\ge1$ is a sequence of estimators in $\mathbb{R}^d$ and there exists a sequence $r_n, n\ge1$ such that $r_n(\widehat{\theta}_n - \theta_0)\overset{d}{\to}W$ for some $W\in\MCH$. Then
    \[
    \lim_{n\to\infty} \Omedbias_P(\widehat{\theta}_n; \theta_0) = \texttt{OMB}(\mathcal{L}(W)) = 2^{d-1}\left(\frac{1}{2^d} - \min_{A\in\mathcal{O}^d}\,\mathbb{P}(W\in A)\right)_+.
    \]
    Moreover, 
    \begin{equation}\label{eq:finite-sample-bound-OMB}
    \left|\Omedbias_P(\widehat{\theta}_n; \theta_0) - \texttt{OMB}(\mathcal{L}(W))\right| ~\le~ 2^{d}\sup_{A\in\mathcal{O}^d}\,|\mathbb{P}(r_n(\widehat{\theta}_n - \theta_0) \in A) - \mathbb{P}(W\in A)|.
    \end{equation}
\end{proposition}
See Section~\ref{appsec:proof-of-prop-convergece-in-distribution-OMB} for a proof. The supremum on the right hand side can be bounded by the uniform distance (multivariate Kolmogorov-Smirnov distance) between the distribution functions of $D(\widehat{\theta}_n - \theta_0)$ and $D W$, where $D$ is a diagonal matrix with all diagonal entries in $\{-1, 1\}$. Such bounds can be obtained for several $M$-estimators using either the classical multivariate Berry--Esseen bounds~\citep{bergstrom1945central} or the high-dimensional central limit theorem~\citep{belloni2018high}. It is worth mentioning that Theorem 4.2 of~\cite{rao1962relations} proves that the convergence in distribution ot $r_n(\widehat{\theta}_n - \theta_0)$ to $W$ for some $W\in\MCH$ implies that the supremum on the right hand side of~\eqref{eq:finite-sample-bound-OMB} converges to zero.
\section{Orthant Median Regularity and Valid Inference}\label{sec:OMR-valid-inference}
The results in the previous sections show that if we have estimators that have an asymptotically zero orthant median bias about $\theta_0$, then there exists asymptotically valid confidence regions. If the estimators are asymptotically bounded, then the confidence regions is also asymptotically bounded. In this section, we show that the converse is also true. Namely, we show that if there exists asymptotically valid non-trivial confidence regions, then there exists a non-trivial estimator which has an asymptotic orthant median bias of zero about $\theta_0$. This is an extension of the univariate result of~\cite{kuchibhotla2023median}. 

Before stating the formal result, let us clarify what we mean by non-trivial confidence region and estimator. We call an estimator $\widehat{\theta}_n$ non-trivial if $\mathbb{P}(\cup_{j=1}^d\{|e_j^{\top}\widehat{\theta}_n| 
< \infty\}) > 0$, i.e., at least one of the coordinates is finite with a positive probability. We call a confidence region $\widehat{\mathrm{CI}}$ non-trivial if $\mathbb{P}(\cup_{j=1}^d\{\Lambda_j < \infty\}) > 0$, where $\Lambda_j$ is the Lebesgue measure of $\{e_j^{\top}x:\,x\in\widehat{\mathrm{CI}}\}$.
\begin{theorem}\label{thm:necessity-of-OMB}
    Suppose there exists a $\gamma\in(0, 1)$ such that $\widehat{\mathrm{CI}}_{\gamma}$ is a non-trivial confidence region for $\theta_0$ of miscoverage at most $\gamma$, i.e.,
    \[
    \mathbb{P}(\theta_0\notin\widehat{\mathrm{CI}}_{\gamma}) \le \gamma.
    \]
    Then there exists a (randomized) estimator $\widehat{\theta}$ such that
    \[
    \Omedbias(\widehat{\theta}; \theta_0) ~\le~ \frac{\gamma}{2}.
    \]
\end{theorem}
See Section~\ref{appsec:proof-of-thm-necessity-of-OMB} for a proof. It is interesting to note that there are no assumptions about the data for Theorem~\ref{thm:necessity-of-OMB}. In particular, we do not need $\widehat{\mathrm{CI}}_{\gamma}$ to be constructed based on independent data. In particular, Theorem~\ref{thm:necessity-of-OMB} applies to dependent data as well. Theorem~\ref{thm:necessity-of-OMB} does not readily prove that an asymptotically valid confidence interval necessitates the existence of an asymptotically zero orthant median bias estimator, unless we can take $\gamma$ to converge to zero with sample size. When the confidence intervals are constructed from independent data, then we can take $\gamma$ converging to zero as $n\to\infty$. The following discussion based on the results of~\cite{kuchibhotla2023median} shows how one can construct a confidence region of arbitrarily high confidence using a confidence region of fixed confidence.

Suppose, for some $\gamma\in[0, 1]$, we have a confidence region procedure that takes $m$ independent observations as input and generates a confidence region $\widehat{\mathrm{CI}}_{m,\gamma}$ such that 
\[
\mathbb{P}(\theta_0\notin\widehat{\mathrm{CI}}_{m,\gamma}) \le \gamma + s_m,
\]
for some sequence $s_m\to0$. Consider the following scheme to construct a new confidence region that has a miscoverage bounded by $\alpha$ for arbitrary $\alpha \le \gamma$. Split the data of $n$ independent observations into $B_{\alpha,\gamma} = \lceil\log_{\gamma}(\alpha)\rceil$ batches of equal sizes $m = \lfloor n/B_{\alpha,\gamma}\rfloor$ (ignore some observations, if needed). Input each batch of $m$ observations through the confidence region procedure to obtain $B_{\alpha,\gamma}$ confidence regions as $\widehat{\mathrm{CI}}_{m,\gamma}^{(k)}, 1\le k\le B_{n,\gamma}$. Set
\[
\widetilde{\mathrm{CI}}_{n,\alpha} = \bigcup_{k=1}^{B_{\alpha,\gamma}} \widehat{\mathrm{CI}}_{n,\gamma}^{(k)}.
\]
Now, observe that
\begin{align*}
\mathbb{P}\left(\theta_0\notin\widetilde{\mathrm{CI}}_{n,\alpha}\right) &= \prod_{k=1}^{B_{\alpha,\gamma}} \mathbb{P}(\theta_0\notin\widehat{\mathrm{CI}}_{m,\gamma}) \le (\gamma + s_m)^{B_{\alpha,\gamma}} = \gamma^{B_{\alpha,\gamma}}(1 + s_m/\gamma)^{B_{\alpha,\gamma}} \le \alpha\exp(B_{\alpha,\gamma}s_{m}/\gamma).
\end{align*}
For any fixed $\alpha\in[0, 1]$, the right hand side converges to $\alpha$ and implies that $\widetilde{\mathrm{CI}}_{n,\alpha}$ is asymptotically valid at level $\alpha$. Following the proof of Theorem 1 of~\cite{kuchibhotla2023median}, we get that there exists a sequence $\alpha_n\to0$ as $n\to0$ such that $\mathbb{P}(\theta_0\notin\widetilde{\mathrm{CI}}_{n,\alpha_n}) \le \alpha_n\to0$. Applying Theorem~\ref{thm:necessity-of-OMB} with $\gamma = \alpha_n$ and the confidence region $\widetilde{\mathrm{CI}}_{n,\alpha_n}$ would imply the existence of an estimator that has an asymptotic orthant median bias of zero. However, we may not know $\alpha_n$ in practice and hence, the resulting confidence interval is not computable only based on the data at hand. Once again following the monotonization trick in the proof of (b)$\Rightarrow$(c) of Theorem 1 of~\cite{kuchibhotla2023median} yields a computable version of $\widetilde{\mathrm{CI}}_{n,\alpha_n}$ and proves the necessity of a an asymptotically zero orthant median bias estimator.  
\section{Discussion}\label{sec:conclusions}
We have provided different notions of multivariate median bias and proposed the rectangular hull confidence regions for general parameters. If the multivariate estimator is suitably median unbiased, then the rectangular hull has valid coverage. In particular, we have focused our study on orthant median bias which measures how far the arrangement of probabilities on orthants differs from that of the standard multivariate normal distributions. We proved that the rectangular hull of $B$ independent estimators with zero orthant median bias has a valid coverage of $1 - \alpha$ if and only if $B \ge 1 - \log_2(1 - (1 - \alpha)^{1/d})$. It is note worthy that this threshold behaves like $\log_2(d/\alpha)$ as $d/\alpha$ diverges. Interestingly, this is precisely the number of estimators used for rectangular hull in~\cite{kuchibhotla2021hulc} under a weaker notion of multivariate median bias, which we referred to as rectilinear median bias. Most importantly, we have extended a result on the necessity of existence of asymptotically median unbiased estimators to the multivariate setting by proving that there exists an asymptotically valid confidence region for $\theta_0$ if and only if there exists an asymptotically orthant median unbiased estimator for $\theta_0$.

There are numerous interesting directions to explore. Firstly, orthant median bias does not converge to zero for many commonly used multivariate estimators. Any estimator that converges (when properly normalized) to a multivariate normal distribution with a non-diagonal covariance matrix has a non-zero orthant median bias, even asymptotically. The two other notions of median bias we discussed, namely, rectilinear median bias and Tukey median bias are both zero for any zero mean normal distribution and are more suitable for commonly used estimators. The ``right'' version of confidence region from independent estimators with zero rectilinear/Tukey median bias is under investigation. In this regard, it may be worth mentioning that~\cite{hoel1961confidence} considered optimal bounds for miscoverage of rectangular hull under controlled rectilinear median bias albeit for $d \le 8$. 
Secondly, many of the bounds presented in this paper are sub-optimal because they do not reduce to some of the optimal bounds we know for $d = 1$. Although of mostly theoretical interest, it would be interesting to obtain sharp bounds for miscoverage of rectangular hull of independent random variables. Finally, there are more interesting notions of multivariate or infinite-dimensional median biases one can define by considering probabilities of cones. We have some preliminary evidence to suggest the necessity of the existence of asymptotically median unbiased estimators for asymptotically valid confidence regions continues to hold to for these general notions of multivariate median bias.  
\bibliography{references}

\begin{thebibliography}{}

\bibitem[Andrews and Guggenberger, 2010]{andrews2010asymptotic}
Andrews, D.~W. and Guggenberger, P. (2010).
\newblock Asymptotic size and a problem with subsampling and with the m out of n bootstrap.
\newblock {\em Econometric Theory}, 26(2):426--468.

\bibitem[Bahadur and Savage, 1956]{bahadur1956nonexistence}
Bahadur, R.~R. and Savage, L.~J. (1956).
\newblock The nonexistence of certain statistical procedures in nonparametric problems.
\newblock {\em The Annals of Mathematical Statistics}, 27(4):1115--1122.

\bibitem[Basu et~al., 2012]{basu2012projection}
Basu, R., Bhattacharya, B.~B., and Talukdar, T. (2012).
\newblock The projection median of a set of points in $\mathbb{R}^d$.
\newblock {\em Discrete \& Computational Geometry}, 47(2):329--346.

\bibitem[Belloni et~al., 2018]{belloni2018high}
Belloni, A., Chernozhukov, V., Chetverikov, D., Hansen, C., and Kato, K. (2018).
\newblock High-dimensional econometrics and regularized gmm.
\newblock {\em arXiv preprint arXiv:1806.01888}.

\bibitem[Bergstr{\"o}m, 1945]{bergstrom1945central}
Bergstr{\"o}m, H. (1945).
\newblock On the central limit theorem in the space ${R}^k, k> 1$.
\newblock {\em Scandinavian Actuarial Journal}, 1945(1-2):106--127.

\bibitem[Bertail et~al., 1999]{bertail1999subsampling}
Bertail, P., Politis, D.~N., and Romano, J.~P. (1999).
\newblock On subsampling estimators with unknown rate of convergence.
\newblock {\em Journal of the American Statistical Association}, 94(446):569--579.

\bibitem[DasGupta, 2008]{dasgupta2008asymptotic}
DasGupta, A. (2008).
\newblock {\em Asymptotic theory of statistics and probability}, volume 180.
\newblock Springer.

\bibitem[Hirano and Porter, 2012]{hirano2012impossibility}
Hirano, K. and Porter, J.~R. (2012).
\newblock Impossibility results for nondifferentiable functionals.
\newblock {\em Econometrica}, 80(4):1769--1790.

\bibitem[Hoel and Scheuer, 1961]{hoel1961confidence}
Hoel, P. and Scheuer, E. (1961).
\newblock Confidence sets for multivariate medians.
\newblock {\em The Annals of Mathematical Statistics}, pages 477--484.

\bibitem[Kuchibhotla et~al., 2021]{kuchibhotla2021hulc}
Kuchibhotla, A.~K., Balakrishnan, S., and Wasserman, L. (2021).
\newblock The {HulC}: Confidence regions from convex hulls.
\newblock {\em arXiv preprint arXiv:2105.14577}.

\bibitem[Kuchibhotla et~al., 2023]{kuchibhotla2023median}
Kuchibhotla, A.~K., Balakrishnan, S., and Wasserman, L. (2023).
\newblock Median regularity and honest inference.
\newblock {\em Biometrika}, 110(3):831--838.

\bibitem[Oja, 2013]{Oja2013}
Oja, H. (2013).
\newblock {\em Multivariate Median}, pages 3--15.
\newblock Springer Berlin Heidelberg, Berlin, Heidelberg.

\bibitem[Pfanzagl, 1998]{pfanzagl1998nonexistence}
Pfanzagl, J. (1998).
\newblock The nonexistence of confidence sets for discontinuous functionals.
\newblock {\em Journal of statistical planning and inference}, 75(1):9--20.

\bibitem[Rao, 1962]{rao1962relations}
Rao, R.~R. (1962).
\newblock Relations between weak and uniform convergence of measures with applications.
\newblock {\em The Annals of Mathematical Statistics}, pages 659--680.

\bibitem[Serfling, 2006]{serfling2006multivariate}
Serfling, R.~J. (2006).
\newblock Multivariate symmetry and asymmetry.
\newblock {\em Encyclopedia of statistical sciences}, 8:5338--5345.

\bibitem[Shao and Tu, 2012]{shao2012jackknife}
Shao, J. and Tu, D. (2012).
\newblock {\em The jackknife and bootstrap}.
\newblock Springer Science \& Business Media.

\bibitem[Sung, 1990]{sung1990generalized}
Sung, N. (1990).
\newblock A generalized cram{\'e}r-rao analogue for median-unbiased estimators.
\newblock {\em Journal of multivariate analysis}, 32(2):204--212.

\bibitem[Tibshirani and Wasserman, 1988]{tibshirani1988sensitive}
Tibshirani, R. and Wasserman, L.~A. (1988).
\newblock Sensitive parameters.
\newblock {\em Canadian Journal of Statistics}, 16(2):185--192.

\bibitem[Zhu et~al., 2020]{zhu2020does}
Zhu, B., Jiao, J., and Steinhardt, J. (2020).
\newblock When does the tukey median work?
\newblock In {\em 2020 IEEE International Symposium on Information Theory (ISIT)}, pages 1201--1206. IEEE.

\bibitem[Zuo and Serfling, 2000]{zuo2000performance}
Zuo, Y. and Serfling, R. (2000).
\newblock On the performance of some robust nonparametric location measures relative to a general notion of multivariate symmetry.
\newblock {\em Journal of Statistical Planning and Inference}, 84(1-2):55--79.

\end{thebibliography}
\bibliographystyle{apalike}

\newpage
\setcounter{section}{0}
\setcounter{equation}{0}
\setcounter{figure}{0}
\renewcommand{\thesection}{S.\arabic{section}}
\renewcommand{\theequation}{E.\arabic{equation}}
\renewcommand{\thefigure}{A.\arabic{figure}}
\renewcommand{\theHsection}{S.\arabic{section}}
\renewcommand{\theHequation}{E.\arabic{equation}}
\renewcommand{\theHfigure}{A.\arabic{figure}}
  \begin{center}
  \Large {\bf Supplement to ``Rectangular Hull Confidence Regions for Multivariate Parameters''}
  \end{center}
       
\begin{abstract}
This supplement contains the proofs of all the main results in the paper, some supporting lemmas, and additional simulations. 
\end{abstract}
\section{Relation between zero median bias and symmetry}\label{appsec:med-bias-symmetry}
\begin{proposition}\label{prop:med-bias-symmetry}
    The following statements are true for any estimator $\widehat{\theta}_n$ and $\theta_0$.
    \begin{enumerate}
        \item $\Rmedbias_P(\widehat{\theta}_n; \theta_0) = 0$ if the marginal distributions of $\widehat{\theta}_n - \theta_0$ are symmetric around zero (i.e., $e_j^{\top}(\widehat{\theta}_n - \theta_0) \overset{d}{=} e_j^{\top}(\theta_0 - \widehat{\theta}_n)$).
        \item $\Tmedbias_P(\widehat{\theta}_n; \theta_0) = 0$ if $\widehat{\theta}_n - \theta_0$ is either angularly symmetric or half-space symmetric or central symmetric. (Angular symmetry means $(\widehat{\theta}_n - \theta_0)/\|\widehat{\theta}_n - \theta_0\| \overset{d}{=} -(\widehat{\theta}_n - \theta_0)/\|\widehat{\theta}_n - \theta_0\|$, half-space symmetry means $\lambda^{\top}(\widehat{\theta}_n - \theta_0) \overset{d}{=} -\lambda^{\top}(\widehat{\theta}_n - \theta_0)$ for all $\lambda\in\mathbb{R}^d$, and finally central symmetry means $\widehat{\theta}_n - \theta_0 \overset{d}{=} \theta_0 - \widehat{\theta}_n$.)
        \item $\Omedbias_P(\widehat{\theta}_n; \theta_0) = 0$ if $\widehat{\theta}_n - \theta_0$ is sign symmetric (i.e., $\widehat{\theta}_n - \theta_0 \overset{d}{=} D(\widehat{\theta}_n - \theta_0)$ for all diagonal matrices $D$ with diagonal entries of unit absolute value).
    \end{enumerate}
\end{proposition}
\begin{proof}
\begin{enumerate}
    \item Because $e_j^{\top}(\widehat{\theta}_n - \theta_0) \overset{d}{=} -e_j^{\top}(\widehat{\theta}_n - \theta_0)$, we have
    \begin{equation}\label{eq:implication-marginal-symmetry}
    \mathbb{P}(e_j^{\top}(\widehat{\theta}_n - \theta_0) \le 0) = \mathbb{P}(e_j^{\top}(\widehat{\theta}_n - \theta_0) \ge 0).
    \end{equation}
    Moreover, we know
    \[
    \mathbb{P}(e_j^{\top}(\widehat{\theta}_n - \theta_0) \le 0) + \mathbb{P}(e_j^{\top}(\widehat{\theta}_n - \theta_0) \ge 0) = 1 + \mathbb{P}(e_j^{\top}(\widehat{\theta}_n - \theta_0) = 0) \ge 1.
    \]
    From~\eqref{eq:implication-marginal-symmetry}, the left hand side above equals $2\mathbb{P}(\eta e_j^{\top}(\widehat{\theta}_n - \theta_0))$ for $\eta\in\{-1, 1\}$. Therefore, 
    \[
    \mathbb{P}(\eta e_j^{\top}(\widehat{\theta}_n - \theta_0)) \ge \frac{1}{2},\quad\mbox{for all}\quad \eta\in\{-1, 1\}, 1\le j \le d.
    \]
    The result follows.
    \item 
    Half-space symmetry implies that $\lambda^{\top}(\widehat{\theta}_n - \theta_0) \overset{d}{=} -\lambda^{\top}(\widehat{\theta}_n - \theta_0)$. Following the proof of part 1, this readily implies
    \[
    \mathbb{P}(\eta\lambda^{\top}(\widehat{\theta}_n - \theta_0) \ge 0) \ge \frac{1}{2},\quad\mbox{for all}\quad \eta\in\{-1, 1\}, \lambda\in\mathbb{R}^d.
    \]
    This implies the result under half-space symmetry. To prove the result under central and angular symmetry, first note that under either of the conditions,
    \[
    \mathbb{P}(\lambda^{\top}(\widehat{\theta}_n - \theta_0) \le 0) = \mathbb{P}(-\lambda^{\top}(\widehat{\theta}_n - \theta_0) \ge 0).
    \]
    (This is not to say that angular or central symmetry implies half-space symmetry because we are only considering the probability around zero.) Now following the proof of part 1 implies the result. 
    \item Sign symmetry implies that
    \begin{equation}\label{eq:implication-sign-symmetry}
    \mathbb{P}(\widehat{\theta}_n - \theta_0 \in A) = \mathbb{P}(\widehat{\theta}_n - \theta_0\in B)\quad\mbox{for all}\quad A, B\in\mathcal{O}^d.
    \end{equation}
    To see this note that for any orthant $A$, there exists a vector $\eta_A\in\{-1, 1\}^s$ such that $A =\{x\in\mathbb{R}^d:\,\eta_{A,j}x_j \ge 0\}$. Hence, we can write
    \begin{align*}
    \mathbb{P}(\widehat{\theta}_n - \theta_0 \in B) &= \mathbb{P}(\eta_{B,j}(\widehat{\theta}_{n,j} - \theta_{0,j}) \ge 0\mbox{ for }1\le j\le d)\\
    &= \mathbb{P}\left(\frac{\eta_{B,j}}{\eta_{A,j}}\eta_{A,j}(\widehat{\theta}_{n,j} - \theta_{0,j}) \ge 0\mbox{ for }1\le j\le d\right)\\
    &= \mathbb{P}(D_{A,B}(\widehat{\theta}_n - \theta_0) \in A),
    \end{align*}
    where $D_{A,B} := \mbox{diag}(\eta_{B,j}/\eta_{A,j}, 1\le j\le d)$. Finally, because $D_{A,B}(\widehat{\theta}_n - \theta_0) \overset{d}{=} \widehat{\theta}_n - \theta_0$, we get~\eqref{eq:implication-sign-symmetry}.
    Now observe that
    \[
    \mathbb{R}^d = \bigcup_{A\in\mathcal{O}^d} A \quad\Rightarrow\quad 1 \le \sum_{A\in\mathcal{O}^d}\mathbb{P}(A) = 2^d\mathbb{P}(A) \quad\Rightarrow\quad \mathbb{P}(A) \ge \frac{1}{2^d},\quad\mbox{for all}\quad A\in\mathcal{O}^d.
    \]
    Hence, the result follows.
    
\end{enumerate}
\end{proof}

\section{Proof of Proposition~\ref{prop:relation-med-biases}}\label{appsec:proof-of-prop-relation-med-bias}
To show that the rectilinear median bias is bounded above by the orthant median bias, it suffices to show that for all $1\le j\le d$, 
\[
\mathbb{P}\left(\eta_je_j^{\top}(\widehat{\theta}_n - \theta_0) \le 0\right) \ge 1/2 - \Omedbias_P(\widehat{\theta}_n; \theta_0) \quad\mbox{for}\quad \eta_j\in\{-1, 1\}.
\]
Because $\widehat{\theta}_n - \theta_0$ has a continuous distribution, it suffices to consider $\eta_j = 1$. Without loss of generality, we fix $j = 1$. Note that for any $x\in\mathbb{R}^d$, 
\[
\{e_1^{\top}x \le 0\} = \bigcup_{s\in \{-1, 1\}^d}\{x_1 \le 0, s_j x_j \le 0,\, 2 \le j\le d\}.
\]
The union on the right hand side has $2^{d-1}$ sets. For any two sets in the union on the right hand side, the intersection is a hyperplane in the $d-1$ dimensional space. By the continuity assumption on $\widehat{\theta}_n - \theta_0$, such intersections have zero probability. Therefore, 
\[
\mathbb{P}\left(e_1^{\top}(\widehat{\theta}_n - \theta_0) \le 0\right) = \sum_{s\in\{-1, 1\}^d} \mathbb{P}\left(\widehat{\theta}_{n,1} - \theta_{0,1} \le 0,\, s_j(\widehat{\theta}_{n,j} - \theta_{0,j}) \le 0,\,2 \le j\le d\right).
\]
Orthant median bias for a continuous distribution means
\[
\mathbb{P}(\widehat{\theta}_n - \theta_0 \in A) \ge \frac{1}{2^d} - \frac{1}{2^{d-1}}\Omedbias_P(\widehat{\theta}_n; \theta_0),\quad\mbox{for all}\quad A\in\mathcal{O}^d.
\]
This implies that
\begin{align*}
\mathbb{P}\left(e_1^{\top}(\widehat{\theta}_n - \theta_0) \le 0\right) &\ge 2^{d-1}\left(\frac{1}{2^d} - \frac{1}{2^{d-1}}\Omedbias_P(\widehat{\theta}_n; \theta_0)\right)\\
&= \frac{1}{2} - \Omedbias_P(\widehat{\theta}_n; \theta_0).
\end{align*}
The same calculation yields the same inequality for $1\le j\le d$ and implies the result.
\section{Proof of Proposition~\ref{prop:MCH-characterization}}\label{appsec:proof-of-prop-MCH-characterization}
Note that
\[
\{x\in\mathbb{R}^d:\,\|x\|_0 < d\} ~=~ \bigcup_{j = 1}^d \{e_j^{\top}x = 0\}.
\]
The probability measure $Q$ belongs to $\MCH$ if and only if the probability of this union is zero. This holds if and only if $Q(\{x:\,e_j^{\top}x = 0\}) = 0$ for all $1\le j\le d$. This completes the proof.
\section{Proof of Theorem~\ref{thm:miscoverage-MCH}}\label{appsec:proof-of-thm-miscoverage-MCH}
We first present a few preliminary lemmas that are needed for our proof of Theorem~\ref{thm:miscoverage-MCH}.


\begin{lemma}\label{lmm:group-spread-ineq}
For any $N>0$, $0 \le \beta \le \frac{1}{N}$, suppose $x_1,\cdots, x_N \ge \beta$ such that $\sum_{i=1}^N x_i = 1$. Then for any $B \ge 1$, $$N \left(\frac{1}{N}\right)^B \le \sum_{i=1}^N x_i^B \le (N-1)\beta^B + \left(1-(N-1)\beta\right)^B$$
\end{lemma}
\begin{proof}By Jensen's inequality on the convex function $g(z) = z^B$, we have
\begin{align*}
    &\frac{1}{N}\sum_{i=1}^N{x_i^B} \geq \left[\frac{1}{N}\sum_{i=1}^N x_i\right]^B\\
    &\Rightarrow \sum_{i=1}^N{x_i^B} \geq N \left(\frac{1}{N}\right)^B \;\;\;\text{ since } \sum_{i=1}^N x_i = 1.
\end{align*} 
thus establishing the first inequality. This is sharp because $(1/N, \ldots, 1/N)\in\mathcal{S} := \{(x_1, \ldots, x_N)\in[\beta, 1]^N:\, \sum_{i=1}^N x_i = 1\}$. 

For the second inequality, note that $(x_1, \ldots, x_N)\mapsto \sum_{i=1}^N x_i^B$ is a convex function and $\mathcal{S}$ is also a convex set. Therefore, the maximum of $\sum_{i=1}^N x_i^B$ on $\mathcal{S}$ is attained on the boundary of $\mathcal{S}$. Because $\mathcal{S}$ is the convex hull of 
\[
\mathcal{S}' := \{(x_1, \ldots, x_N)\in[\beta, 1]^N:\, x_j = 1 - (N-1)\beta,\mbox{ for some }j\mbox{ and }x_{j'} = \beta\mbox{ for all }j'\neq j\},
\]
it follows that the maximum value is attained at one of these points. The function value is the same as $(N-1)\beta^B + (1 - (N-1)\beta)^B$ at any of the points in $\mathcal{S}'$, we get the second inequality. Evidently, it is also sharp. 

\end{proof}
\begin{proof}[Proof of Theorem~\ref{thm:miscoverage-MCH}]
Note that
\begin{align*}
\{\theta_0\notin\RHull_B\} &= \bigcup_{j=1}^d \left\{e_j^{\top}\theta_{0}\notin\left[\min_{1\le k\le B}e_j^{\top}\widehat{\theta}^{(k)},\,\max_{1\le k\le B}e_j^{\top}\widehat{\theta}^{(k)}\right]\right\}.
\end{align*}
By the inclusion-exclusion principle, we obtain
\begin{align*}
    \mathbb{P}(\theta_0\notin\RHull_B) &= \sum_{\ell = 1}^d (-1)^{\ell - 1}\sum_{I\subseteq[d], |I| = \ell} \mathbb{P}\left(\bigcap_{j\in I}\left\{e_j^{\top}\theta_{0}\notin\left[\min_{1\le k\le B}e_j^{\top}\widehat{\theta}^{(k)},\,\max_{1\le k\le B}e_j^{\top}\widehat{\theta}^{(k)}\right]\right\}\right).
\end{align*}
The intersection event is that for all $j\in I,$ either $e_j^{\top}\theta_0$ is less than all of $e_j^{\top}\widehat{\theta}^{(k)}, 1\le k\le B$ or is larger than all of $e_j^{\top}\widehat{\theta}^{(k)}, 1\le k\le B$. This implies that for any $I\subseteq[n]$ with $|I| = \ell$,
\begin{align*}
    &\bigcap_{j\in I}\left\{e_j^{\top}\theta_{0}\notin\left[\min_{1\le k\le B}e_j^{\top}\widehat{\theta}^{(k)},\,\max_{1\le k\le B}e_j^{\top}\widehat{\theta}^{(k)}\right]\right\}\\
    &= \bigcup_{\eta\in\{-1, 1\}^{\ell}} \left(\bigcap_{j\in I, 1\le k\le B}\left\{\eta_j(e_j^{\top}\widehat{\theta}^{(k)} - e_j^{\top}\theta_0) < 0\right\}\right).
\end{align*}
This is a disjoint union, and hence, 
\begin{equation}\label{eq:expression-rectangular-hull-coverage}
\begin{split}
    \mathbb{P}(\theta_0\notin\RHull_B) &= \sum_{\ell = 1}^d (-1)^{\ell - 1}\sum_{\substack{I\subseteq[d], |I| = \ell,\\\eta\in\{-1, 1\}^I}} \mathbb{P}\left(\bigcap_{j\in I, 1\le k\le B}\left\{\eta_j(e_j^{\top}\widehat{\theta}^{(k)} - e_j^{\top}\theta_0) < 0\right\}\right)\\
    &= \sum_{\ell = 1}^d (-1)^{\ell - 1}\sum_{\substack{I\subseteq[d], |I| = \ell,\\\eta\in\{-1, 1\}^I}} \prod_{k=1}^B \mathbb{P}\left(\bigcap_{j\in I}\left\{\eta_j(e_j^{\top}\widehat{\theta}^{(k)} - e_j^{\top}\theta_0) < 0\right\}\right)\\
    &= \sum_{\ell = 1}^d (-1)^{\ell - 1}\sum_{\substack{I\subseteq[d], |I| = \ell,\\\eta\in\{-1, 1\}^I}} \left(\mathbb{P}\left(\bigcap_{j\in I}\left\{\eta_j(e_j^{\top}\widehat{\theta}^{(1)} - e_j^{\top}\theta_0) < 0\right\}\right)\right)^B\\
    &= \sum_{\ell = 1}^d (-1)^{\ell - 1}\sum_{I\subseteq[d], |I| = \ell}\left[\sum_{\eta\in\{-1, 1\}^I} \left(\mathbb{P}\left(\mathcal{C}_{I,\eta}\right)\right)^B\right].
    \end{split}
\end{equation} 
where 
\[
\mathcal{C}_{I,\eta} = \bigcap_{j\in I}\left\{\eta_j(e_j^{\top}\widehat{\theta}^{(1)} - e_j^{\top}\theta_0) < 0\right\} = \left\{\widehat{\theta}^{(1)}-\theta_0 \in \bigcap_{j\in I}\{x \in \mathbb{R}^d: \eta_j e_j^{\top}x < 0\} \right\}.
\]

To prove the result, it now suffices to upper and lower bound the probabilities on the right hand side. Following a similar argument to the one used in the proof of Proposition~\ref{prop:relation-med-biases}, for any $I\subseteq[d], |I| = \ell$, we can write $\cap_{j\in I}\{x: \eta_j e_j^{\top}x \le 0\}$ as a union of $2^{d-\ell}$ many orthants which have zero-probability intersections. This implies that
\begin{align*}
\mathbb{P}\left( \mathcal{C}_{I,\eta} \right)= \mathbb{P}\left(\widehat{\theta}^{(1)}-\theta_0 \in \bigcap_{j\in I}\{x: \eta_j e_j^{\top}x \le 0\}\right) &\ge 2^{d-\ell}\min_{A\in\mathcal{O}^d}\mathbb{P}\left(\widehat{\theta}^{(1)} - \theta_0\in A\right)\\ 
&\ge 2^{d-\ell}\left(\frac{1}{2^d} - \frac{\delta}{2^{d-1}}\right) = \frac{1}{2^{\ell}} - \frac{\delta}{2^{\ell - 1}}.
\end{align*}
Now, notice that 
\[
    \bigcup_{\eta \in \{-1,1\}^I} \mathcal{C}_{I,\eta} = \left\{\widehat{\theta}^{(1)}-\theta_0 \in \bigcup_{\eta \in \{-1,1\}^I} \bigcap_{j\in I}\{x: \eta_j e_j^{\top}x \le 0\} \right\} = \left\{\widehat{\theta}^{(1)}-\theta_0 \in \mathbb{R}^d \right\},
\]
which is a probability one event, and because  $\widehat{\theta}^{(1)}$ does not equal to $\theta_0$ almost surely for any subset of coordinates due the condition of MCH, events $\mathcal{C}_{I,\eta}$ intersect with zero probability, and we get that $\sum_{\eta\in\{-1, 1\}^I} \mathbb{P}\left(\mathcal{C}_{I,\eta}\right) = 1.$

Thus with $N = 2^{\ell}$, $\beta = {(1-2\delta)}/{2^{\ell}}$, and for all $\eta \in \{-1,1\}^I$, $x_{\eta} = \mathbb{P}\left(\mathcal{C}_{I,\eta}\right)$, we apply Lemma~\ref{lmm:group-spread-ineq} to get 
\begin{align*}
   2^{\ell} \left(\frac{1}{2^\ell}\right)^B &\le \sum_{\eta\in\{-1, 1\}^I} \left(\mathbb{P}\left(\mathcal{C}_{I,\eta}\right)\right)^B \le (2^{\ell} - 1) \left(\frac{1}{2^\ell} - \frac{\delta}{2^{\ell-1}}\right)^B + \left(\frac{1}{2^\ell} + (2^\ell - 1)\cdot\frac{\delta}{2^{\ell-1}}\right)^B\\
   \Rightarrow \binom{d}{\ell} 2^{\ell} \left(\frac{1}{2^\ell}\right)^B &\le \sum_{I\subseteq[d], |I| = \ell}\sum_{\eta\in\{-1, 1\}^I} \left(\mathbb{P}\left(\mathcal{C}_{I,\eta}\right)\right)^B\\ 
   &\le \binom{d}{\ell} (2^{\ell} - 1) \left(\frac{1}{2^\ell} - \frac{\delta}{2^{\ell-1}}\right)^B + \left(\frac{1}{2^\ell} + (2^\ell - 1)\cdot\frac{\delta}{2^{\ell-1}}\right)^B,
\end{align*} 
which when combined with~\eqref{eq:expression-rectangular-hull-coverage} yields the desired result.
\end{proof}

\section{Proof of Proposition~\ref{prop:properties-of-L-and-U}}\label{appsec:proof-of-prop-properties-of-L-and-U}
From the definition of $U(B,\delta;d)$, it is clear that it is a differentiable function of $\delta$.
Define
\[
g_j(\delta) := (2^j - 1)\left(\frac{1}{2^j} - \frac{\delta}{2^{j-1}}\right)^B + \left(\frac{1}{2^j} + (2^j - 1)\frac{\delta}{2^{j-1}}\right)^B.
\]
Clearly, $g_j(0) = 2^j(1/2^j)^B$ and
\begin{align*}
\frac{d}{d\delta}g_j(\delta) &= -B(2^j - 1)\left(\frac{1}{2^j} - \frac{\delta}{2^{j-1}}\right)^{B-1}\frac{1}{2^{j-1}} + B\left(\frac{1}{2^j} + (2^j - 1)\frac{\delta}{2^{j-1}}\right)^{B-1}\frac{2^{j} - 1}{2^{j-1}}\\
&= \frac{B(2^j - 1)}{2^{j - 1}}\left[\left(\frac{1}{2^j} + (2^j - 1)\frac{\delta}{2^{j-1}}\right)^{B-1} - \left(\frac{1}{2^j} - \frac{\delta}{2^{j-1}}\right)^{B-1}\right] \ge 0,
\end{align*}
for all $\delta \in [0, 1/2].$ Therefore, $\delta\mapsto g_j(\delta)$ is an increasing function and hence, $U(B, \delta; d)$ is an increasing function of $\delta$.

Note that
\[
L(B, 0; d) = U(B, 0; d) = \sum_{j=1}^d (-1)^{j-1}\binom{d}{j}2^j\left(\frac{1}{2^j}\right)^{B} = 1 - (1 - 2^{-B + 1})^d.
\]

From the definition of $U(B, \delta; d)$, 
\[
\frac{d}{d\delta}U(B, \delta; d) = \sum_{\substack{1\le j\le d,\\j\,\mathrm{odd}}} \binom{d}{j}\frac{d}{d\delta}g_j(\delta) \quad\Rightarrow\quad \frac{d}{d\delta}U(B,\delta;d)\bigg|_{\delta = 0} = 0.
\]
Moreover,
\begin{align*}
&\frac{d^2}{d\delta^2}U(B, \delta; d)\\ 
&= \sum_{\substack{1\le j\le d,\\j\,\mathrm{odd}}} \binom{d}{j}\frac{d^2}{d\delta^2}g_j(\delta)\\
&= \sum_{\substack{1\le j\le d,\\j\,\mathrm{odd}}} \binom{d}{j}\frac{B(B-1)(2^j - 1)}{2^{j-1}}\left[\left(\frac{1}{2^j} + (2^j - 1)\frac{\delta}{2^{j-1}}\right)^{B-2}\frac{2^j - 1}{2^{j-1}} + \frac{1}{2^{j-1}}\left(\frac{1}{2^j} - \frac{\delta}{2^{j-1}}\right)^{B-2}\right]\\
&\le \sum_{\substack{1\le j\le d,\\j\,\mathrm{odd}}} \binom{d}{j}\frac{B(B-1)(2^{j} - 1)}{2^{j-1}}\left(\frac{1}{2^j} + (2^j - 1)\frac{\delta}{2^{j-1}}\right)^{B-2}\\
&\le \sum_{\substack{1\le j\le d,\\j\,\mathrm{odd}}} \binom{d}{j}\frac{B(B-1)(2^{j} - 1)}{2^{j-1}} \le 2B(B-1)\sum_{\substack{1\le j\le d,\\j\,\mathrm{odd}}} \binom{d}{j} = 2^dB(B-1),
\end{align*}
because $\delta\in[0, 1/2]$.
\section{Proof of Theorem~\ref{thm:asymptotic-validity}}\label{appsec:proof-of-asymp-validity}
The proof essentially follows from Theorem~\ref{thm:miscoverage-MCH}. From Theorem~\ref{thm:miscoverage-MCH}, it follows that
\[
\mathbb{P}(\theta_0\notin\widehat{\mathrm{CI}}_{\alpha}) \le \mathbb{E}[U(B^*, \delta_{n/B_{\alpha,d}}; d)].
\]
By a Taylor series expansion around $0$, we get from Proposition~\ref{prop:properties-of-L-and-U} that
\[
U(B^*, \delta_{n/B_{\alpha,d}}; d) \le U(B^*, 0; d) + 2^{d-1}B_{\alpha,d}(B_{\alpha,d} - 1)\delta_{n/B_{\alpha,d}}^2.
\]
By definition of $B^*$, $\mathbb{E}[U(B^*, 0; d)] = \alpha$. This completes the proof of second inequality. The proof of the first inequality is exactly the same except now one uses the properties for $L(\cdot, \cdot; d)$.
\section[Proof of ]{Proof of~\eqref{eq:miscoverage-prob-MCH-non-MCH}}\label{appsec:proof-of-miscoverage-prob-MCH-non-MCH}
Recall from the ~\eqref{eq:expression-rectangular-hull-coverage}, the miscoverage probability of the parameter(origin) by the rectangular hull of $B$ i.i.d. estimators is given by $$\mathbb{P}(0\notin\RHull_B) = \sum_{\ell = 1}^d (-1)^{\ell - 1}\sum_{I\subseteq[d], |I| = \ell}\left[\sum_{\eta\in\{-1, 1\}^I} \left(\mathbb{P}\left(\mathcal{C}_{I,\eta}\right)\right)^B\right]$$ where 
$$\mathcal{C}_{I,\eta} = \bigcap_{j\in I}\left\{\eta_j e_j^{\top}\widehat{\theta}^{(1)} < 0\right\} = \left\{\widehat{\theta}^{(1)} \in \bigcap_{j\in I}\{x \in \mathbb{R}^d: \eta_j e_j^{\top}x < 0\} \right\}.
$$\\

Here $d=2$, $B=3$, and we are considering the two probability measures $\mathbb{P}_{\mu_{\texttt{M}}}$ and $\mathbb{P}_{\mu_{\texttt{NM}}}$ as described in \textit{Section 3.2}. So we compute the following probabilities:

\begin{align*}
    \mathbb{P}_{\mu_{\texttt{M}}}(\mathcal{C}_{\{1,2\},(-1,-1)}) ~&=~ \mu_{\texttt{M}}(\{x_1 > 0, x_2 > 0\}) ~=~  0.2,\\
    \mathbb{P}_{\mu_{\texttt{M}}}(\mathcal{C}_{\{1,2\},(1,-1)}) ~&=~ \mu_{\texttt{M}}(\{x_1 < 0, x_2 > 0\}) ~=~  0.2,\\
    \mathbb{P}_{\mu_{\texttt{M}}}(\mathcal{C}_{\{1,2\},(1,1)}) ~&=~ \mu_{\texttt{M}}(\{x_1 < 0, x_2 < 0\}) ~=~  0.2,\\
    \mathbb{P}_{\mu_{\texttt{M}}}(\mathcal{C}_{\{1,2\},(-1,1)}) ~&=~ \mu_{\texttt{M}}(\{x_1 > 0, x_2 > 0\}) ~=~  0.4,\\
    \mathbb{P}_{\mu_{\texttt{M}}}(\mathcal{C}_{\{1\},-1}) ~&=~ \mu_{\texttt{M}}(\{x_1 > 0\}) ~=~  0.6,\\
    \mathbb{P}_{\mu_{\texttt{M}}}(\mathcal{C}_{\{1\},1}) ~&=~ \mu_{\texttt{M}}(\{x_1 < 0\}) ~=~  0.4,\\
    \mathbb{P}_{\mu_{\texttt{M}}}(\mathcal{C}_{\{2\},-1}) ~&=~ \mu_{\texttt{M}}(\{x_2 > 0\}) ~=~  0.4,\\
    \mathbb{P}_{\mu_{\texttt{M}}}(\mathcal{C}_{\{2\},1}) ~&=~ \mu_{\texttt{M}}(\{x_2 < 0\}) ~=~  0.6.
\end{align*}

Thus from ~\eqref{eq:expression-rectangular-hull-coverage}, $$\mathbb{P}_{\mu_{\texttt{M}}}(0\notin\RHull_3) = 2(0.6)^3 + 2(0.4)^3 - 3(0.2)^3 - (0.4)^3 = 0.472$$

Similar calculations reveal
\begin{align*}
    \mathbb{P}_{\mu_{\texttt{NM}}}(\mathcal{C}_{\{1,2\},(-1,-1)}) ~&=~ \mu_{\texttt{NM}}(\{x_1 > 0, x_2 > 0\}) ~=~  0.1,\\
    \mathbb{P}_{\mu_{\texttt{NM}}}(\mathcal{C}_{\{1,2\},(1,-1)}) ~&=~ \mu_{\texttt{NM}}(\{x_1 < 0, x_2 > 0\}) ~=~  0,\\
    \mathbb{P}_{\mu_{\texttt{NM}}}(\mathcal{C}_{\{1,2\},(1,1)}) ~&=~ \mu_{\texttt{NM}}(\{x_1 < 0, x_2 < 0\}) ~=~  0.1,\\
    \mathbb{P}_{\mu_{\texttt{NM}}}(\mathcal{C}_{\{1,2\},(-1,1)}) ~&=~ \mu_{\texttt{NM}}(\{x_1 > 0, x_2 > 0\}) ~=~  0.6,\\
    \mathbb{P}_{\mu_{\texttt{NM}}}(\mathcal{C}_{\{1\},-1}) ~&=~ \mu_{\texttt{NM}}(\{x_1 > 0\}) ~=~  0.7,\\
    \mathbb{P}_{\mu_{\texttt{NM}}}(\mathcal{C}_{\{1\},1}) ~&=~ \mu_{\texttt{NM}}(\{x_1 < 0\}) ~=~  0.2,\\
    \mathbb{P}_{\mu_{\texttt{NM}}}(\mathcal{C}_{\{2\},-1}) ~&=~ \mu_{\texttt{NM}}(\{x_2 > 0\}) ~=~  0.2,\\
    \mathbb{P}_{\mu_{\texttt{NM}}}(\mathcal{C}_{\{2\},1}) ~&=~ \mu_{\texttt{NM}}(\{x_2 < 0\}) ~=~  0.7.
\end{align*} giving us $$\mathbb{P}_{\mu_{\texttt{NM}}}(0\notin\RHull_3) = 2(0.7)^3 + 2(0.2)^3 - (0.6)^3 - 2(0.1)^3 = 0.484$$ completing the proof.
\section{Proof of Lemma~\ref{lem:miscoverage-elementary-operation}}\label{appsec:proof-of-miscoverage-elementary-operation}
To prove the result, we first argue that the event of miscovering zero by the rectangular hull can be written only in terms of the sign vectors. Note that $0$ does not belong in the rectangular hull if and only if there exists a coordinate index $j\in\{1, 2, \ldots, d\}$ such that all of $e_j^{\top}X_i, 1\le i\le B$ are either positive or negative. Equivalently, there exists a coordinate index $j$ such that either all of $e_j^{\top}\mbox{sign}(X_i), 1\le i\le B$ are $-1$ or $1$. Therefore, 
\begin{align*}
&\mathbb{P}\left(0 \notin \bigotimes_{j=1}^d\left[\min_{1\le i\le B}e_j^{\top}X_i,\,\max_{1\le i\le B}e_j^{\top}X_i\right]\right)\\
&= \mathbb{P}\left(\bigcup_{j=1}^d \bigcup_{\eta\in\{-1, 1\}}\bigcap_{i=1}^B \left\{e_j^{\top}\mbox{sign}(X_i) = \eta\right\}\right)\\ 
&= \mathbb{P}\left(0 \notin \bigotimes_{j=1}^d\left[\min_{1\le i\le B}e_j^{\top}\mbox{sign}(X_i),\,\max_{1\le i\le B}e_j^{\top}\mbox{sign}(X_i)\right]\right).
\end{align*}
This implies that the miscoverage probability is the same for all probability measures $\mu$ such that the induced sign probability measure is $\mu_{\texttt{sgn}}$ and moreover, the miscoverage event can be written in terms of the sign vectors. Let $(s_1, \ldots, s_B)$ and $(p_1, \ldots, p_B)$ represent the sign vectors of $(X_1, \ldots, X_B)$ and $(Y_1, \ldots, Y_B)$, respectively. 

We now define a coupling through which we can generate $p_1, \ldots, p_B$ iid from $\nu_{\texttt{sgn}}$ given a sample $s_1, \ldots, s_B$ iid from $\mu_{\texttt{sgn}}$, and show that
\begin{equation}\label{eq:miscoverage-inclusion}
\left\{0 \notin \bigotimes_{j=1}^d \left[\min_{1\le i\le B}e_j^{\top}s_i,\,\max_{1\le i\le B}e_j^{\top}s_i\right]\right\} \subseteq \left\{0\notin \bigotimes_{j=1}^d \left[\min_{1\le i\le B}e_j^{\top}p_i,\,\max_{1\le i\le B}e_j^{\top}p_i\right]\right\},
\end{equation}
for such a coupling. This will complete the proof.

Let the parameters of the elementary operation $\mathcal{E}$ be $q \ge 0, (\eta, \eta')$. If $q = 0$, then there is nothing to prove because $\mu_{\mbox{sgn}}(\gamma) = \nu_{\mbox{sgn}}(\gamma)$ for all $\gamma\in U_d$. Assume $q > 0$. Let $U\sim\mbox{Uniform}[0, 1]$.
For any given $s\sim\mu_{\texttt{sgn}}$, define $p\in U_d$ as
\begin{equation}\label{eq:coupling-definition}
p = 
\begin{cases}
s, &\mbox{if }s \notin\{\eta, \eta'\},\\
\eta, &\mbox{if }s = \eta, U < (\mu_{\texttt{sign}}(\eta) - q)/\mu_{\texttt{sgn}}(\eta),\\
\eta', &\mbox{if }s = \eta, U \ge (\mu_{\texttt{sign}}(\eta) - q)/\mu_{\texttt{sgn}}(\eta),\\
\eta', &\mbox{if }s = \eta'. 
\end{cases}
\end{equation}
Note that $p = s$ if $s\neq \eta$ and $p$ is either $\eta$ or $\eta'$ if $s = \eta$.
Observe that
\begin{enumerate}
    \item For $\gamma\notin\{\eta, \eta'\}$, by~\eqref{eq:elementary-operation-definition},
    \[
    \mathbb{P}(p = \gamma) = \mathbb{P}(s = \gamma) = \mu_{\texttt{sgn}}(\gamma) = \nu_{\texttt{sgn}}(\gamma).
    \]
    \item For $\gamma = \eta$,
    \[
    \mathbb{P}(p = \gamma) = \mathbb{P}\left(s = \eta, U < \frac{\mu_{\texttt{sgn}}(\eta) - q}{\mu_{\texttt{sgn}}(\eta)}\right) = \frac{\mu_{\texttt{sgn}}(\eta) - q}{\mu_{\texttt{sgn}}(\eta)}\mu_{\texttt{sgn}}(\eta) = \mu_{\texttt{sgn}}(\eta) - q = \nu_{\texttt{sgn}}(\eta).
    \]
    \item For $\gamma = \eta'$,
    \[
    \mathbb{P}(p = \gamma) = \mathbb{P}\left(s = \eta'\right) + \mathbb{P}\left(s = \eta, U \ge \frac{\mu_{\texttt{sgn}}(\eta) - q}{\mu_{\texttt{sgn}}(\eta)}\right) = \mu_{\texttt{sgn}}(\eta') + q = \nu_{\texttt{sgn}}(\eta').
    \]
\end{enumerate}
Therefore, $p$ as defined in~\eqref{eq:coupling-definition} has $\nu_{\texttt{sgn}}$ as the probability measure. Now, consider the event of miscoverage of zero by the rectangular hull:
\[
\left\{0 \notin \bigotimes_{j=1}^d \left[\min_{1\le i\le B}e_j^{\top}s_i,\,\max_{1\le i\le B}e_j^{\top}s_i\right]\right\} = \bigcup_{j=1}^d \left\{0\notin\left[\min_{1\le i\le B}e_j^{\top}s_i,\,\max_{1\le i\le B}e_j^{\top}s_i\right]\right\}.
\]
To show that the inclusion~\eqref{eq:miscoverage-inclusion} holds true, it suffices to verify 
\begin{equation}\label{eq:miscoverage-inclusion-j}
\left\{0\notin\left[\min_{1\le i\le B}e_j^{\top}s_i,\,\max_{1\le i\le B}e_j^{\top}s_i\right]\right\} \subseteq \left\{0\notin\left[\min_{1\le i\le B}e_j^{\top}p_i,\,\max_{1\le i\le B}e_j^{\top}p_i\right]\right\}\quad\mbox{for all}\quad 1\le j\le d.
\end{equation}
Set $J = \{j:\,\eta_j\neq\eta_j'\}$. For $j\notin J$, $\eta_j = \eta_j'$ which implies $e_j^{\top}s_i = e_j^{\top}p_i$ by~\eqref{eq:coupling-definition}. This implies that the inclusion~\eqref{eq:miscoverage-inclusion-j} holds true for $j\notin J$. To verify~\eqref{eq:miscoverage-inclusion-j} for $j\in J$, note that if $s_i \notin \{\eta, \eta'\}$ for all $1\le i\le B$, then $s_i = p_i$ for all $1\le i\le B$ and inclusion~\eqref{eq:miscoverage-inclusion-j} holds true. If there exists at least one $i\in\{1, 2, \ldots, B\}$ such that $s_i = \eta$, then $e_j^{\top}s_i = \eta_j = 0$ for $j\in J$ and hence, the left hand side of~\eqref{eq:miscoverage-inclusion-j} is an empty set. This implies that the inclusion~\eqref{eq:miscoverage-inclusion-j} holds. If $s_i = \eta'$ for some $i\in\{1, 2, \ldots, B\}$, then by~\eqref{eq:coupling-definition} $p_i = \eta'$ and hence, if none of the $s_i$'s are $\eta$, then $s_i = p_i$ for all $1\le i\le B$ which implies~\eqref{eq:miscoverage-inclusion-j}. This completes the proof of~\eqref{eq:miscoverage-inclusion} and the result.
\section{Proof of Theorem~\ref{thm:general-miscoverage-MCH-validity}}\label{appsec:proof-of-general-miscoverage-MCH-validity}
Let $\mathcal{L}(\widehat{\theta}^{(j)} - \theta_0)$ be $\mu$. If $\mu\in\MCH$, then Theorem~\ref{thm:asymptotic-validity} completes the proof. If $\mu\notin\MCH$, then consider any $\nu_{\varepsilon}$ such that $\mu\preceq_{\texttt{MCH}}\nu_{\varepsilon}$ and 
\[
\texttt{OMB}(\nu_{\varepsilon}) \le (1 + \varepsilon)\Omedbias_P(\widehat{\theta}^{(j)}; \theta_0) \le (1 + \varepsilon)\delta_{n/B_{\alpha,d}}.
\]
From Lemma~\ref{lem:miscoverage-elementary-operation} and Theorem~\ref{thm:asymptotic-validity}, it follows that
\[
\mathbb{P}_{\mu}(\theta_0\notin\widehat{\mathrm{CI}}_{\alpha}) \le \mathbb{P}_{\nu_{\varepsilon}}(\theta_0\notin\widehat{\mathrm{CI}}_{\alpha}) \le \mathbb{E}[U(B^*, (1 + \varepsilon)\delta_{n/B_{\alpha,d}}; d)] \le \alpha + 2^{d-1}B_{\alpha,d}(B_{\alpha,d} - 1)(1 + \varepsilon)^2\delta^2_{n/B_{\alpha,d}}.
\]
Letting $\varepsilon\downarrow0$, we get the result.
\section{Proof of Proposition~\ref{prop:orthant-median-bias-null-est}}\label{appsec:proof-of-orthant-median-bias-null-est}
Without loss of generality, fix $\theta_0 = 0$. As mentioned, the result follows by noting that the degenerate measure at zero can be converted to an MCH distribution with equal probabilities on each orthants which has an orthant median bias of zero. Consider the elementary operations $\mathcal{E}_{\gamma}, \gamma\in V_d$ that takes moves $1/2^d$ from $\eta = (0, 0, \ldots, 0)$ to $\eta' = \gamma$. If $\mu$ is the probability measure of degenerate at $0$ distribution, then $\nu = (\prod_{\gamma\in V_d}\mathcal{E}_{\gamma})\mu$ is an MCH distribution and $\texttt{OMB}(\nu) = 0$. Therefore, $\Omedbias_P(\widehat{\theta}_n; \theta_0) = 0$.
\section{Proof of Proposition~\ref{prop:orthant-median-bias-1-d}}\label{appsec:proof-of-prop-orthant-median-bias-1-d}
The proof of Proposition~\ref{prop:orthant-median-bias-1-d} is very similar to those in the example illustrating elementary operations in $\mathbb{R}$. If $\widehat{\theta}_n$ does not place any mass at $\theta_0$, then $\widehat{\theta}_n - \theta_0\in\MCH$ and the result is trivially true. If $\mathbb{P}(\widehat{\theta}_n = \theta_0) = q > 0$, then setting
\[
\mathbb{P}(\widehat{\theta}_n < \theta_0) = p\quad\mbox{and}\quad\mathbb{P}(\widehat{\theta}_n > \theta_0) = r,
\]
we note that any chain of elementary operations yields an MCH distribution with mass $(p+q_1)$ for $(-\infty, \theta_0)$ and mass $(r+q -q_1)$ for $(\theta_0, \infty),$ where $q_1 \in [0, q]$. It is easy to verify that these are only possible MCH distributions that can be attained starting from $\widehat{\theta}_n - \theta_0$. For such an MCH distribution, the orthant median bias is $(1/2 - \min\{p+q_1, r + q - q_1\})_+$. 

{\color{black} We now claim that $\min_{q_1 \in [0,q]}(1/2 - \min\{p+q_1, r + q - q_1\})_+ = (1/2 - \min\{p+q,r+q\})_+$.
Firstly if $p>1/2$, $\min\{p+q_1, r + q - q_1\} = r + q - q_1$, hence minimization occurs at $q_1=0$, resulting in $(1/2 - \min\{p+q_1, r + q - q_1\})_+ = (1/2 - \min\{p,r+q\})_+ = (1/2 - \min\{p+q,r+q\})_+$. Similarly, if $r>\frac{1}{2}$, the minimization occurs at $q_1=q$, yielding $(1/2 - \min\{p+q_1, r + q - q_1\})_+ = (1/2 - \min\{p+q,r\})_+ = (1/2 - \min\{p+q,r+q\})_+$. Finally if $p,r \leq 1/2$, then $q = (1/2 - p) + (1/2 - r)$, so take $q_1 = 1/2 - p$. Then $(1/2 - \min\{p+q_1, r + q - q_1\})_+ = (1/2 - \min\{1/2, 1/2\})_+ = 0 = (1/2 - \min\{p+q,r+q\})_+$, thus completing the proof.} 


\section{Proof of Proposition~\ref{prop:convergence-in-distribution-OMB}}\label{appsec:proof-of-prop-convergece-in-distribution-OMB}
Let $\mu_n$ be the probability measure of $\widehat{\theta}_n - \theta_0$. Let $\mu$ be the probability measure of $W$. Set
\[
\delta_n = \sup_{A\in\mathcal{O}^d}\,|\mu_n(A) - \mu(A)|.
\]
Let $\nu_{n,\varepsilon}$ be any MCH measure such that $\mu_n\preceq_{\texttt{MCH}}\nu_{n,\varepsilon}$ and
\[
\texttt{OMB}(\nu_{n,\varepsilon}) \le (1 + \varepsilon)\Omedbias_P(\widehat{\theta}_n; \theta_0).
\]
It follows that for all $A\in\mathcal{O}^d$, $\nu_{n,\varepsilon}(A) \le \mu_n(A)$. This implies that 
\[
\texttt{OMB}(\nu_{n,\varepsilon}) = 2^{d-1}\left(\frac{1}{2^d} - \min_{A\in\mathcal{O}^d}\nu_{n,\varepsilon}(A)\right)_+ \ge 2^{d-1}\left(\frac{1}{2^d} - \min_{A\in\mathcal{O}^d} \mu(A) - \delta_n\right)_+ \ge \texttt{OMB}(\mu) - 2^{d-1}\delta_n.
\]
Therefore, letting $\varepsilon\downarrow0$, we get $\Omedbias_P(\widehat{\theta}_n; \theta_0) \ge \texttt{OMB}(\mu) - 2^{d-1}\delta_n$.

To prove the upper bound on the orthant median bias, note that for any orthant $A\in\mathcal{O}^d$, $\nu_{n,\varepsilon}(A) \ge \mu_{n}(A^\circ),$ where $A^\circ$ represents the interior of $A$. This implies $\nu_{n,\varepsilon}(A) \ge \mu_n(A) - \mu_n(\partial A) \ge \mu(A) - \delta_n - \mu_n(\partial A)$ where $\partial A = A\setminus A^\circ$. If $\eta\in V_d$ represents $A$, i.e., $\mbox{sign}(x) = \eta$ for all $x\in A^\circ$, then this is same as
\[
\nu_{n,\varepsilon}(A) \ge \mu(A) - \delta_n - \mu_n(\{x:\, e_j^{\top}\mbox{sign}(x)\eta_j \ge 0\}\setminus\{\eta\}).
\]
On the other hand, note that we have $|\mu_n(A^c) - \mu(A^c)| \le \delta_n$, for any $A\in\mathcal{O}^d$, where $A^c$ represents the complement of $A$. If $\eta_A\in\{-1, 1\}^d$ represents $A$, let $\mathcal{B}_{A}$ represent the set of all edges of $A$, i.e., 
\[
\mathcal{B}_{A} = \{\eta'\in U_d = \{-1, 0, 1\}^d:\, \eta'_j\eta_{A,j} \ge 0\mbox{ for all }1\le j\le d\mbox{ with equality for at least one }j\}.
\]
Recall $U_d = \{-1, 0, 1\}^d$ and $V_d = \{-1, 1\}^d$. With this notation, we can write
\[
A^c = \left(\bigcup_{\gamma\in V_d\setminus\{\eta_A\}} \{x:\,\mbox{sign}(x) = \gamma\}\right)\cup\left(\bigcup_{\gamma\in U_d\setminus (V_d\cup \mathcal{B}_{A})} \{x:\,\mbox{sign}(x) = \gamma\}\right).
\]
Therefore,
\begin{align*}
\mu_n(A^c) &= \sum_{\gamma\in V_d} \mu_{n,\texttt{sgn}}(\gamma) - \mu_{n,\texttt{sgn}}(\eta_A) + \sum_{\gamma\in U_d\setminus V_d} \mu_{n,\texttt{sgn}}(\gamma) - \mu_{n,\texttt{sgn}}(\mathcal{B}_A)\\
&= \sum_{B\in\mathcal{O}^d} \mu_n(B) - \sum_{B\in\mathcal{O}^d} \mu_{n,\texttt{sgn}}(\mathcal{B}_B) - \mu_{n}(A) + \sum_{\gamma\in U_d\setminus V_d} \mu_{n,\texttt{sgn}}(\gamma).
\end{align*}
Summing this equation over all $A\in\mathcal{O}^d$, this implies that
\[
\sum_{A\in\mathcal{O}^d} \mu_n(A^c) = (2^d-1)\sum_{A\in\mathcal{O}^d} \mu_n(A) + 2^d\sum_{\gamma\in U_d\setminus V_d} \mu_{n,\texttt{sgn}}(\gamma) - 2^d\sum_{A\in\mathcal{O}^d} \mu_{n,\texttt{sgn}}(\mathcal{B}_{A}).
\]
Because an $\gamma\in U_d\setminus V_d$ can act as an edge for $2^{d - \|\gamma\|_0}$ orthants, we get that
\[
\sum_{A\in\mathcal{O}^d} \mu_{n,\texttt{sgn}}(\mathcal{B}_A) = \sum_{\gamma\in U_d\setminus V_d} 2^{d-\|\gamma\|_0}\mu_{n,\texttt{sgn}}(\gamma).
\]
Therefore, 
\[
\sum_{A\in\mathcal{O}^d} \mu_n(A^c) = (2^d - 1)\sum_{A\in\mathcal{O}^d} \mu_n(A) - \sum_{\gamma\in U_d\setminus V_d} \left\{2^{2d-\|\gamma\|_0} - 2^d\right\} \mu_{n,\texttt{sgn}}(\gamma).
\]
Because $d - \|\gamma\|_0 \ge 1$ for all $\gamma\in U_d\setminus V_d$, we have $2^{2d - \|\gamma\|_0} \ge 2^{d+1}$ for all $d \ge 1$. Applying the same equation for $\mu$ (and using $\mu\in\MCH$), a simple manipulation yields
\begin{align*}
\sum_{\gamma\in U_d\setminus V_d} \{2^{2d - \|\gamma\|_0} - 2^d\}\mu_{n,\texttt{sgn}}(\gamma) &\le (2^d-1)\sum_{A\in\mathcal{O}^d}|\mu_n(A) - \mu(A)| + \sum_{A\in\mathcal{O}^d}|\mu_n(A^c) - \mu(A^c)|\\ 
&\le 2^d\sum_{A\in\mathcal{O}^d} |\mu_n(A) - \mu(A)| \le 2^d\delta_n.
\end{align*}
This implies that
\[
\mu_{n,\texttt{sgn}}(\mathcal{B}_A) \le \frac{2^d}{2^{d + 1} - 2^d}\delta_n \le \delta_n\quad\mbox{for all}\quad A\in\mathcal{O}^d.
\]
This, in turn, implies that $\nu_{n,\varepsilon}(A) \ge \mu(A) - \delta_n - 2^{d-1}\delta_n$. Therefore, letting $\varepsilon\downarrow0$, 
$$\Omedbias_P(\widehat{\theta}_n; \theta_0) \le \texttt{OMB}(\mu) + (1 + 2^{d-1})\delta_n \le \texttt{OMB}(\mu) + 2^d\delta_n.$$

\section{Proof of Theorem~\ref{thm:necessity-of-OMB}}\label{appsec:proof-of-thm-necessity-of-OMB}
    Let $\widehat{R}_{\gamma}$ be the rectangular hull of $\widehat{\mathrm{CI}}_{\gamma}$. Non-triviality of $\widehat{\mathrm{CI}}_{\gamma}$ implies non-triviality of $\widehat{R}_{\gamma}$. Moreover, we have $\widehat{R}_{\gamma}\supseteq\widehat{\mathrm{CI}}_{\gamma}$ and hence, 
    \[
    \mathbb{P}(\theta_0\notin\widehat{R}_{\gamma}) \le \mathbb{P}(\theta_0 \notin \widehat{\mathrm{CI}}_{\gamma}) \le \gamma.
    \]
    Let us write $\widehat{R}_{\gamma} = \prod_{j=1}^d [\widehat{\ell}_{j,\gamma},\,\widehat{u}_{j,\gamma}]$. Define, for independent random variables $U_1, \ldots, U_d\sim\mbox{Uniform}(0,1)$,
    \[
    \widetilde{R}_{\gamma} = \prod_{j=1}^d [\widehat{\ell}_{j,\gamma} - U_j\Delta_j,\,\widehat{u}_{j,\gamma} + U_j\Delta_j],\quad\mbox{where}\quad \Delta_j = \delta\max\{\widehat{u}_{j,\gamma} - \widehat{\ell}_{j,\gamma},\,\delta\},
    \]
    for any arbitrary $\delta > 0$. (One can let $\delta$ go to zero with the data to ensure that $\widetilde{R}_{\gamma}$ is not much larger than $\widehat{R}_{\gamma}$.) Clearly, $\widetilde{R}_{\gamma}\supseteq\widehat{R}_{\gamma}$ and hence, $\widetilde{R}_{\gamma}$ is also a valid confidence region.

    Define the randomized estimator $\widehat{\theta}$ by letting it be each of the $2^d$ vertices of $\widetilde{R}_{\gamma}$ with probability $1/2^d$. Formally, define $\widehat{\theta}$ by
    \[
    \mathbb{P}\left(e_j^{\top}\widehat{\theta} = \widehat{\ell}_{j,\gamma} - U_j\Delta_j\big|\widetilde{R}_{\gamma}\right) = \mathbb{P}\left(e_j^{\top}\widehat{\theta} = \widehat{u}_{j,\gamma} + U_j\Delta_j\big|\widetilde{R}_{\gamma}\right) = \frac{1}{2},
    \]
    and $e_1^{\top}\widehat{\theta}, \ldots, e_d^{\top}\widehat{\theta}$ (conditional on $\widetilde{R}_{\gamma}$) are independent. 

    Marginally, $e_j^{\top}\widehat{\theta}$ is a continuous random variable for any $j\in\{1, 2, \ldots, d\}$ because of the involvement of the independent uniform random variable $U_j$. Therefore, $\widehat{\theta} - \theta_0\in\MCH$. This implies that 
    \[
    \Omedbias(\widehat{\theta}; \theta_0) = 2^{d-1}\left(\frac{1}{2^d} - \min_{A\in\mathcal{O}^d}\mathbb{P}(\widehat{\theta} - \theta_0\in A)\right)_+.
    \]
    We now argue about the probabilities of orthants. Without loss of generality, consider $A = \{x:\,x_1 \ge 0, x_2 \ge 0, \ldots, x_d \ge 0\}$. 
    \begin{align*}
    \mathbb{P}(\widehat{\theta} - \theta_0\in A) &= \mathbb{E}[\mathbb{P}(\widehat{\theta} - \theta_0 \in A\big|\widetilde{R}_{\gamma})]\\
    &= \mathbb{E}\left[\prod_{j=1}^d\mathbb{P}\left(e_j^{\top}(\widehat{\theta} - \theta_0) \ge 0\big|\widetilde{R}_{\gamma}\right)\right]\\
    &= \mathbb{E}\left[\prod_{j=1}^d\left(\frac{1}{2}\mathbb{P}(\widehat{\ell}_{j,\gamma} - U_j\Delta_j \ge \theta_0|\widehat{R}_{\gamma}) + \frac{1}{2}\mathbb{P}(\widehat{u}_{j,\gamma} + U_j\Delta_j \ge \theta_0|\widehat{R}_{\gamma})\right)\right]\\
    &\ge \mathbb{E}\left[\frac{1}{2^d}\prod_{j=1}^d \left(\mathbf{1}\{\widehat{\ell}_{j,\gamma} \ge \theta_0\} + \mathbf{1}\{\widehat{u}_{j,\gamma} \ge \theta_0\}\right)\right]\\
    &\ge \frac{1}{2^d}\mathbb{P}\left(\bigcap_{j=1}^d \{\widehat{u}_{j,\gamma} \ge \theta_0\}\right) \ge \frac{1}{2^d}\mathbb{P}(\theta_0\in\widehat{R}_{\gamma}) \ge \frac{1 - \gamma}{2^d}.
    \end{align*}
    Similar calculations hold for all orthants and yield
    \[
    \Omedbias(\widehat{\theta}; \theta_0) \le 2^{d-1}\left(\frac{1}{2^d} - \frac{1 - \gamma}{2^d}\right) = \frac{\gamma}{2}.
    \]
    This completes the proof.
\end{document}